\documentclass[reqno]{amsart}

\usepackage{geometry}
\usepackage{fancyhdr}
\usepackage{setspace}
\usepackage{mathtools}
\usepackage{amsfonts}
\usepackage{a4wide}
\usepackage{textcomp}
\usepackage{epsfig}
\usepackage{tikz-cd}

\usepackage[utf8]{inputenc}
\usepackage[T1]{fontenc}
\usepackage{lmodern}
\usepackage{amsmath}
\usepackage{amsthm}
\usepackage{amssymb}
\usepackage{upgreek}
\usepackage{enumerate}
\usepackage{mathrsfs}
\usepackage{mathscinet}
\usepackage{stmaryrd} 
\usepackage{graphicx}
\usepackage{xcolor}
\usepackage[all,cmtip]{xy}
\usepackage{url}

\usepackage{aliascnt}

\usepackage{hyperref}

\newcommand\mathens[1]{\mathbb{#1}} 

\newcommand{\Z}{\mathens{Z}}

\newcommand{\R}{\mathens{R}}

\newcommand{\id}{\mathrm{id}}
\newcommand{\nlmas}\upmu

\newcommand\tpsi{\widetilde{\psi}}

\DeclareMathOperator{\gr}{gr}

\DeclareMathOperator{\osc}{osc}
\DeclareMathOperator{\dis}{dis}
\DeclareMathOperator{\Cont}{Cont}

\DeclareMathOperator{\can}{can}
\DeclareMathOperator{\Ham}{Ham}

\DeclareMathOperator{\Crit}{Crit}

\DeclareMathOperator\ind{ind}

\DeclareFontFamily{U}{mathb}{\hyphenchar\font45}
\DeclareFontShape{U}{mathb}{m}{n}{
      <5> <6> <7> <8> <9> <10> gen * mathb
      <10.95> mathb10 <12> <14.4> <17.28> <20.74> <24.88> mathb12
}{}
\DeclareSymbolFont{mathb}{U}{mathb}{m}{n}
\DeclareMathSymbol{\cll}{3}{mathb}{"CE}

\newtheorem{thm}{Theorem}[section]

\newaliascnt{prop}{thm}
\newtheorem{prop}[prop]{Proposition}
\aliascntresetthe{prop}

\newaliascnt{defi}{thm}

\aliascntresetthe{defi}

\newaliascnt{ex}{thm}
\newtheorem{ex}[ex]{Example}
\aliascntresetthe{ex}

\newaliascnt{rmk}{thm}
\newtheorem{rmk}[rmk]{Remark}
\aliascntresetthe{rmk}

\newaliascnt{lemma}{thm}
\newtheorem{lemma}[lemma]{Lemma}
\aliascntresetthe{lemma}

\newaliascnt{cor}{thm}
\newtheorem{cor}[cor]{Corollary}
\aliascntresetthe{cor}

\newaliascnt{conjecture}{thm}

\aliascntresetthe{conjecture}

\newtheorem{prop-def}[thm]{Definition-proposition}

\theoremstyle{definition}

\theoremstyle{remark}

\makeatletter\let\@wraptoccontribs\wraptoccontribs\makeatother

\usepackage{tikz,pgfplots}
\usetikzlibrary{decorations.markings, intersections, arrows.meta, positioning, calc, math, matrix}

\tikzset{middlearrow/.style={
        decoration={markings,
            mark= at position 0.5 with {\arrow{#1}} ,
        },
        postaction={decorate}
    }
}

\begin{document}

\title[]{Spectral selectors on lens spaces
and applications to the geometry of the group of contactomorphisms}

\author[S. Allais]{Simon Allais}
\address{S. Allais, Universit\'e de Strasbourg, IRMA UMR 7501, F-67000 Strasbourg, France}
\email{simon.allais@math.unistra.fr}
\urladdr{https://irma.math.unistra.fr/~allais/}

\author[P.-A. Arlove]{Pierre-Alexandre Arlove}
\address{P.-A. Arlove, Ruhr-Universit\"at Bochum, Fakult\"at f\"ur Mathematik, 
44780 Bochum, Germany}
\email{pierre-alexandre.arlove@ruhr-uni-bochum.de}

\author[S. Sandon]{Sheila Sandon}
\address{S. Sandon, Universit\'e de Strasbourg, CNRS, IRMA UMR 7501,
F-67000 Strasbourg, France}
\email{sandon@math.unistra.fr}

\maketitle

\begin{abstract}
\noindent
Using Givental's non-linear Maslov index
we define a sequence of spectral selectors
on the universal cover of the identity component
of the contactomorphism group
of any lens space.
As applications,
we prove for lens spaces with equal weights
that the standard Reeb flow is a geodesic
for the discriminant and oscillation norms,
and we define for general lens spaces
a stably unbounded conjugation invariant spectral pseudonorm.
\end{abstract}


\section{Introduction}

For any integer $k \geq 2$
and $n$-tuple $\underline{w} = (w_1, \cdots, w_n)$
of positive integers relatively prime to $k$,
the lens space $L_k^{2n-1} (\underline{w})$
is the quotient of the unit sphere $\mathbb{S}^{2n-1}$
in $\mathbb{R}^{2n} \equiv \mathbb{C}^n$
by the free $\mathbb{Z}_k$-action
generated by the map
$$
(z_1, \cdots, z_n) \mapsto
\big( e^{\frac{2 \pi i}{k} \cdot w_1} \, z_1, \cdots, e^{\frac{2 \pi i}{k} \cdot w_n} \, z_n \big) \,.
$$
We endow $L_k^{2n-1} (\underline{w})$
with its canonical contact structure $\xi_0$,
the kernel of the contact form $\alpha_0$ whose pullback $\bar{\alpha}_0$
by the projection $\mathbb{S}^{2n-1} \rightarrow L_k^{2n-1} (\underline{w})$
is equal to the pullback of $\sum_{j = 1}^n x_j dy_j - y_j dx_j$
by the inclusion $\mathbb{S}^{2n-1} \hookrightarrow \mathbb{R}^{2n}$.
We denote by $\widetilde{\Cont_0} \big(L_k^{2n-1} (\underline{w}), \xi_0 \big)$
the universal cover of the identity component
$\Cont_0 (L_k^{2n-1} \big(\underline{w}), \xi_0 \big)$
of the contactomorphism group.
The non-linear Maslov index
is a quasimorphism
$$
\mu: \widetilde{\Cont_0} \big(L_k^{2n-1} (\underline{w}), \xi_0 \big) \rightarrow \mathbb{Z} \,,
$$
defined by Givental \cite{Giv90}
for real projective spaces
and extended to general lens spaces in \cite{GKPS}.
Roughly speaking, it counts with multiplicity
the number of intersections of contact isotopies
with (a certain subspace of) the space of contactomorphisms
that have at least one discriminant point.

Recall that a point $p$
of a contact manifold $(M, \xi)$
is said to be a discriminant point of a contactomorphism $\phi$
if $\phi(p) = p$ and $(\phi^{\ast} \alpha)_p = \alpha_p$
for some (hence any) contact form $\alpha$ for $\xi$,
and is said to be a translated point of $\phi$
with respect to a contact form $\alpha$
if there exists a real number $T$
(in general not unique)
such that $p$ is a discriminant point of $r_{-T}^{\alpha} \circ \phi$,
where $\{r_t^{\alpha}\}$ denotes the Reeb flow;
such $T$ is then said to be a translation
of the translated point $p$.
Discriminant and translated points
play a key role in certain proofs of several global rigidity results
in contact topology,
related in particular to contact non-squeezing
\cite{San11, FSZ, AM},
orderability \cite{EP00, Bhu, San11, San11b, GKPS, AM},
and bi-invariant metrics on the contactomorphism group
\cite{San10, CS, PA}.
In particular,
Givental's non-linear Maslov index for projective spaces
has been used in \cite{EP00}, \cite{San13} and \cite{CS} respectively
to prove that real projective spaces are orderable,
satisfy a contact analogue of the Arnold conjecture
and have unbounded discriminant and oscillation norms.
All these results have then been generalized to lens spaces in \cite{GKPS}
(recovering for orderability a result also obtained in \cite{Milin}
and \cite{San11b}).
In the original article of Givental \cite{Giv90},
the non-linear Maslov index on projective spaces
and a Legendrian version of it
have been applied in particular
to prove the Weinstein and chord conjectures,
and a result on existence of Reeb chords
between Legendrian submanifolds
Legendrian isotopic to each other.
Moreover, an analogue of the non-linear Maslov index for complex projective spaces
has been used by Givental \cite{Giv90} and Th\'eret \cite{Theret-Rotation}
to prove the Arnold conjectures on fixed points of Hamiltonian symplectomorphisms
and Lagrangian intersections.

In the present article we use the non-linear Maslov index
to define spectral selectors on the universal cover
of the identity component of the contactomorphism group of lens spaces,
i.e.\ maps
$$
c_j : \widetilde{\Cont_0} \big(L_k^{2n-1} (\underline{w}), \xi_0 \big) \to \mathbb{R}
$$
that associate to every element of $ \widetilde{\Cont_0} \big(L_k^{2n-1} (\underline{w}), \xi_0 \big)$
a real number belonging to its action spectrum.
Recall that the action spectrum
of a contactomorphism $\phi$ of a contact manifold $(M, \xi)$
with respect to a contact form $\alpha$
is the set $\mathcal{A}_{\alpha} (\phi)$ of real numbers
that are translations of translated points of $\phi$
with respect to $\alpha$.
We denote by
$$
\Pi: \widetilde{\Cont_0} (M, \xi)
\rightarrow \Cont_0 (M, \xi)
$$
the standard projection,
which sends an element
$\widetilde{\phi} = [ \{ \phi_t \}_{t \in [0, 1]} ]$
of $\widetilde{\Cont_0} (M, \xi)$ to $\phi_1$,
and define the action spectrum
of an element $\widetilde{\phi}$ of $\widetilde{\Cont_0} (M, \xi)$
by $\mathcal{A}_{\alpha} (\widetilde{\phi}) = \mathcal{A}_{\alpha} \big( \Pi (\widetilde{\phi}) \big)$.
Let
$$
\mathcal{L}: \widetilde{\Cont_0} \big(L_k^{2n-1} (\underline{w}), \xi_0 \big)
\rightarrow \widetilde{\Cont_0} \big(\mathbb{S}^{2n-1}, \bar{\xi}_0 = \ker (\bar{\alpha}_0)\big)
$$
be the map that sends $\widetilde{\phi} = [ \{ \phi_t \}_{t \in [0, 1]} ]$
to the element of $\widetilde{\Cont_0} (\mathbb{S}^{2n-1}, \bar{\xi}_0)$
represented by the lift of $\{ \phi_t \}_{t \in [0, 1]}$ to $(\mathbb{S}^{2n-1}, \bar{\xi}_0)$.
For $\widetilde{\phi} \in \widetilde{\Cont_0} \big( L_k^{2n-1} (\underline{w}), \xi_0 \big)$
we denote
$$
\mathcal{A} (\widetilde{\phi}) = \mathcal{A}_{\alpha_0} (\widetilde{\phi})
$$
and
$$
\bar{\mathcal{A}} (\widetilde{\phi})
= \mathcal{A}_{\bar{\alpha}_0} \big( \mathcal{L}(\widetilde{\phi}) \big)
\subset \mathcal{A} (\widetilde{\phi}) \,.
$$
The sets $\mathcal{A} (\widetilde{\phi})$ and $\bar{\mathcal{A}} (\widetilde{\phi})$
are invariant by translation by $T_{\underline{w}}$ and $2 \pi$ respectively,
where $T_{\underline{w}}$ denotes the period of the Reeb flow
of $\alpha_0$ on $L_k^{2n-1} (\underline{w})$.
For a real number $T$ we denote
$$
\big\lceil T \big\rceil_{T_{\underline{w}}} = T_{\underline{w}} \, \left\lceil \, \frac{T}{T_{\underline{w}}} \, \right\rceil
\quad \text{ and } \quad
\big\lfloor T \big\rfloor_{T_{\underline{w}}} = T_{\underline{w}} \, \left\lfloor \, \frac{T}{T_{\underline{w}}} \, \right\rfloor \,,
$$
thus $\lceil T \rceil_{T_{\underline{w}}}$ and $\lfloor T \rfloor_{T_{\underline{w}}}$
are respectively the smallest multiple of $T_{\underline{w}}$
greater or equal than $T$
and the greatest multiple of $T_{\underline{w}}$
smaller or equal than $T$.

Before stating our main result
we recall that, since $\big( L_k^{2n-1} (\underline{w}), \xi_0 \big)$ is orderable,
the relation $\leq$ on $\widetilde{\Cont_0} \big( L_k^{2n-1} (\underline{w}), \xi_0 \big)$
defined by posing $\widetilde{\phi} \leq \widetilde{\psi}$
if there is a non-negative contact isotopy representing
$\widetilde{\psi} \,\cdot\, \widetilde{\phi}^{-1}$
is a bi-invariant partial order;
this is the partial order
that appears in point (\ref{monotonicity}) below.
Recall also that a translated point $p$
of a contactomorphism $\phi$ of a contact manifold $(M, \xi)$
with respect to a contact form $\alpha$
is said to be non-degenerate for a translation $T$
if there is no vector $X \in T_pM \smallsetminus \{0\}$
such that $(r^{\alpha}_{-T} \circ \phi)_{\ast} (X) = X$
and $dg (X) = 0$,
where $g$ is the conformal factor of $\phi$,
i.e.\ the function defined by the relation
$\phi^{\ast} \alpha = e^g \alpha$.
In the case of $(\mathbb{S}^{2n-1}, \bar{\xi}_0)$,
if a translated point of a contactomorphism
with respect to $\bar{\alpha}_0$
is non-degenerate for a certain translation
then it is non-degenerate for all the translations;
we then just say that it is non-degenerate.
For any $T \in \mathbb{R}$
we denote
$$
\widetilde{r_T} = [ \{r_{Tt}\}_{t \in [0, 1]} ] \,,
$$
where $\{r_t\}$ is the Reeb flow on $L_k^{2n-1} (\underline{w})$
of $\alpha_0$.
Moreover,
we denote by $\widetilde{\id}$
the identity on $\widetilde{\Cont_0} \big( L_k^{2n-1} (\underline{w}), \xi_0 \big)$.

Our main result is the following theorem.

\begin{thm}[Spectral selectors]\label{theorem: main}
There exists a non-decreasing sequence of maps
$$
\big\{\, c_j : \widetilde{\Cont_0} \big( L_k^{2n-1} (\underline{w}), \xi_0 \big) \to \mathbb{R} \,,\; j \in \mathbb{Z} \,\big\}
$$
satisfying the following properties:
\renewcommand{\theenumi}{\roman{enumi}}
\begin{enumerate}

\item \label{spectrality} \emph{Spectrality}:
$$
c_j (\widetilde{\phi}) \in \bar{\mathcal{A}} (\widetilde{\phi}) \,.
$$

\item \label{normalization} \emph{Normalization}:
$$
c_{0} (\widetilde{\id}) = 0 \,.
$$

\item \label{infinite de points translates} \emph{Relation with translated points}:
if all the translated points of $\Pi \big(\mathcal{L}(\widetilde{\phi})\big)$
with respect to $\bar{\alpha}_0$
are non-degenerate
then the spectral selectors
$\{\, c_j (\widetilde{\phi}) \,, j \in \mathbb{Z} \,\}$
are all distinct.
On the other hand,
if
$$
c_{j - 1} (\widetilde{\phi}) < c_{j} (\widetilde{\phi}) = c_{j + 1} (\widetilde{\phi}) = \cdots = c_{j + m} (\widetilde{\phi})
= T < c_{j + m + 1} (\widetilde{\phi})
$$
for some $j$ and $1 \leq m \leq 2n-1$
and either $k$ is even or $j$ is odd or $m > 1$
then $\Pi \big(\mathcal{L}(\widetilde{\phi})\big)$
has infinitely many translated points of translation $T$
with respect to $\bar{\alpha}_0$.

\item \label{non-degeneracy} \emph{Non-degeneracy}:
if
$$
c_{-2n + 1} (\widetilde{\phi}) = c_0 (\widetilde{\phi}) = 0
$$
then
$\Pi \big( \mathcal{L} (\widetilde{\phi}) \big)$ is the identity.

\item \label{composition with the Reeb flow main} \emph{Composition with the Reeb flow}:
for every $T \in \mathbb{R}$ we have
$$
c_j (\widetilde{r_T} \cdot \widetilde{\phi})
=  c_j (\widetilde{\phi}) + T \,;
$$
in particular,
$c_0 ( \widetilde{r_T} ) = T $.

\item \label{periodicity} \emph{Periodicity}:
$$
c_{j + 2n} (\widetilde{\phi}) = c_j (\widetilde{\phi}) + 2 \pi \,.
$$

\item \label{monotonicity} \emph{Monotonicity}:
if $\widetilde{\phi} \leq \widetilde{\psi}$ then $c_j (\widetilde{\phi}) \leq c_j (\widetilde{\psi})$. 

\item \label{continuity} \emph{Continuity}:
if $\widetilde{\phi} \cdot \widetilde{\psi}^{-1}$ is represented by a contact isotopy
with Hamiltonian function $H_t: L_k^{2n-1} (\underline{w}) \rightarrow \mathbb{R}$
with respect to $\alpha_0$
then
\begin{equation*}
\int_0^1 \min H_t \, dt
\leq c_j (\widetilde{\phi}) - c_j (\widetilde{\psi})
\leq \int_0^1 \max H_t \, dt \,.
\end{equation*}
Moreover,
each $c_j$ is continuous
with respect to the $\mathcal{C}^1$-topology.

\item \label{triangle inequality} \emph{Triangle inequality}:
if either $k$ is even or $j$ is even then
\[
c_{j + l} \big(\widetilde{\phi} \cdot \widetilde{\psi}\big)
\leq c_j \big(\widetilde{\phi}\big) + \left\lceil c_l \big(\widetilde{\psi}\big) \right\rceil_{T_{\underline{w}}} \,,
\]
in particular
$$
\left\lceil c_{j + l} \big( \widetilde{\phi} \cdot \widetilde{\psi} \big) \right\rceil_{T_{\underline{w}}}
\leq \left\lceil c_j \big( \widetilde{\phi} \big) \right\rceil_{T_{\underline{w}}}
+ \left\lceil c_l \big( \widetilde{\psi} \big) \right\rceil_{T_{\underline{w}}} \,.
$$

\item \label{conjugaison invariance} \emph{Conjugation invariance}:
$$
\left\lceil c_j \big(\widetilde{\psi} \cdot \widetilde{\phi} \cdot \widetilde{\psi}^{-1} \big) \right\rceil_{T_{\underline{w}}}
= \left\lceil c_j \big( \widetilde{\phi} \big) \right\rceil_{T_{\underline{w}}} \,.
$$

\item \label{PD} \emph{Poincar\'e duality}:
$$
\left\lceil c_j \big(\widetilde{\phi} \big) \right\rceil_{T_{\underline{w}}}
= - \left\lfloor c_{- j - (2n-1)} \big(\widetilde{\phi}^{-1} \big) \right\rfloor_{T_{\underline{w}}} \,.
$$

\end{enumerate}

\end{thm}

Using the Hamiltonian version of the non-linear Maslov index
for complex projective spaces
(\cite{Giv90, Theret-Rotation, Car13})
it is possible to define also spectral invariants
$$
\big\{\, c_j : \widetilde{\Ham} (\mathbb{CP}^{n}, \omega_0) \to \mathbb{R}
\,,\; j \in \mathbb{Z} \,\big\}
$$
satisfying properties analogue to those
of \autoref{theorem: main},
with stronger statements for (\ref{triangle inequality}),
(\ref{conjugaison invariance}) and (\ref{PD})
not involving the $T_{\underline{w}}$-floors and ceilings.
Such spectral invariants coincide with the ones
defined by the first author in \cite{OnHoferZehnderGF},
and their projections to $\mathbb{S}^1$
coincide with the rotation numbers
defined by Th\'eret in \cite{Theret-Rotation}.
Moreover,
their properties are analogue to those
satisfied by the spectral invariants
defined with Floer homology
by Entov and Polterovich in \cite{EP-Calabiqm}.
The fact that in the contact case
the statements of the triangle inequality, conjugation invariance
and Poincar\'e duality properties
are weaker than in the symplectic case
and involve the $T_{\underline{w}}$-floors and ceilings
is similar to what happens for the spectral selectors
of compactly supported contactomorphisms
of $( \mathbb{R}^{2n} \times \mathbb{S}^1, \xi_0 )$
defined by the third author in \cite{San11}.
Indeed,
these spectral selectors are contact analogues
of the spectral selectors of compactly supported
Hamiltonian symplectomorphisms of $( \mathbb{R}^{2n}, \omega_0 )$
defined by Viterbo in \cite{Vit},
but they satisfy weaker versions of the triangle inequality,
conjugation invariance and Poincar\'e duality properties
involving their (integral) floors and ceilings.
Roughly speaking,
this difference with respect to the symplectic case
can be explained as follows.
If we see $( \mathbb{R}^{2n} \times \mathbb{S}^1, \xi_0 )$
and $\big( L_k^{2n-1} (1, \cdots, 1), \xi_0 \big)$
as prequantizations of $( \mathbb{R}^{2n}, \omega_0 )$
and $(\mathbb{CP}^{n-1}, \omega_0)$ respectively
then in both cases the contact spectral selectors
are generalizations of the symplectic ones,
in the sense that the symplectic spectral selectors
of Hamiltonian isotopies of $( \mathbb{R}^{2n}, \omega_0 )$
and $(\mathbb{CP}^{n-1}, \omega_0)$
coincide with the contact spectral selectors of their lifts
to  $( \mathbb{R}^{2n} \times \mathbb{S}^1, \xi_0 )$
and $\big( L_k^{2n-1} (1, \cdots, 1), \xi_0 \big)$ respectively.
The fact that the contact spectral selectors satisfy weaker versions of the triangle inequality,
conjugation invariance and Poincaré duality properties
involving their floors and ceilings
with respect to the period of the Reeb flow
is due to the fact that,
while the lifts of Hamiltonian isotopies of $( \mathbb{R}^{2n}, \omega_0 )$
and $(\mathbb{CP}^{n-1}, \omega_0)$ 
are exactly the contact isotopies
that commute with the standard Reeb flows,
general contact isotopies
commute with the Reeb flow at time $t$
only when $t$ is a multiple of the period of the Reeb flow.
For conjugation invariance, for instance,
while in the symplectic case
the action spectrum is invariant by conjugation,
in the contact case this is in general not true:
the translated points of a contactomorphism
are in general not in bijection
with those of a conjugation.
However,
if the Reeb flow is periodic
then the translated points of translation
equal to the period of the Reeb flow
are discriminant points,
which are invariant by conjugation,
and this fact can be used to prove
that the corresponding floor and ceiling of the spectral selectors
are invariant by conjugation
(see also the discussion in \cite{San11}).

In \cite{AA} the first and second authors
have defined invariants $c_+$ and $c_-$
for elements of the universal cover
of any closed orderable contact manifold
and for contactomorphisms of any closed contact manifold
with orderable contactomorphism group.
In the universal cover case,
these invariants satisfy
all the properties in \autoref{theorem: main}
(including conjugation invariance
if the Reeb flow is periodic,
and with stronger versions for the triangle inequality
and Poincar\'e duality properties
not involving floors and ceilings)
except for periodicity
(there are only two invariants $c_+$ and $c_-$,
while we have a sequence $c_j$
related by periodicity),
spectrality and (\ref{infinite de points translates}).
These properties are important for us
to obtain the applications discussed below.
In particular,
spectrality is crucial to obtain \autoref{corollary: Reeb geodesic}
and the relation between the pseudonorm $\nu$
of \autoref{corollary: word norms} and the oscillation norm,
while periodicity is used in \autoref{corollary: word norms}
to show that the induced norm $\nu_{\ast}$ is bounded
(see also Remarks \ref{remark: quasimorphisms}
and \ref{remark: contact Arnold conjecture}
below for two more consequences of these properties).
The first and second authors also defined in \cite{AA}
invariants for Legendrian submanifolds
and Legendrian isotopies
(when the involved spaces are orderable)
that do satisfy a spectrality property.
Using the Legendrian version of the non-linear Maslov index
defined in \cite{Giv90}
it should be possible to obtain
also a Legendrian version of our spectral selectors,
with properties similar to those in \autoref{theorem: main}.
However,
as far as we can see,
the only new application of these spectral selectors
with respect to those in \cite{AA}
would be a better lower bound for the number of Reeb chords
between Legendrian submanifolds
Legendrian isotopic to each other,
but (at least in the case of real projective space)
such bound is already given by Givental in \cite{Giv90}
just using the non-linear Maslov index.

\begin{rmk}\label{remark: quasimorphisms}
Properties (\ref{periodicity}), (\ref{triangle inequality}) and (\ref{PD})
imply that each $c_j$ is a quasimorphism.
\end{rmk}

\begin{rmk}\label{remark: contact Arnold conjecture}
Properties (\ref{spectrality}), (\ref{infinite de points translates})
and (\ref{periodicity})
imply that every contactomorphism of $\big( L_k^{2n-1} (\underline{w}), \xi_0 \big)$
contact isotopic to the identity
has at least $n$ translated points
with respect to $\alpha_0$,
and at least $2n$ if either $k$ is even
or all the translated points
are non-degenerate.
We thus recover the corresponding result of \cite{GKPS},
but not the optimal bound obtained by the first author in \cite{lens},
where it is proved
that every contactomorphism of $\big( L_k^{2n-1} (\underline{w}), \xi_0 \big)$
contact isotopic to the identity
has at least $2n$ translated points
with respect to $\alpha_0$.
\end{rmk}

\begin{rmk}\label{remark: stronger relation with translated points k prime}
Suppose that $k$ is prime.
Recall that the cohomological index $\ind (A)$ 
of a subset $A$ of $L_k^{2n-1} (\underline{w})$
is the dimension over $\mathbb{Z}_k$
of the image of the map
$\check{H}^{\ast} \big( L_k^{2n-1} (\underline{w}); \mathbb{Z}_k \big)
\rightarrow \check{H}^{\ast} ( A; \mathbb{Z}_k )$
on \v{C}ech cohomology induced by the inclusion
$A \hookrightarrow L_k^{2n-1} (\underline{w})$.
As we will see,
property (\ref{infinite de points translates})
can be refined in this case as follows:
if 
$$
c_{j - 1} (\widetilde{\phi}) < c_{j} (\widetilde{\phi}) = c_{j + 1} (\widetilde{\phi}) = \cdots = c_{j + m} (\widetilde{\phi})
= T < c_{j + m + 1} (\widetilde{\phi})
$$
for some $j$ and $1 \leq m \leq 2n-1$
then the set of translated points of translation $T$ of $\Pi (\widetilde{\phi})$
has cohomological index greater or equal than $m$,
and greater or equal than $m + 1$
if either $k = 2$ or $j$ is odd.
\end{rmk}

As a first application of Theorem \ref{theorem: main}
we prove that the standard Reeb flow on lens spaces with equal weights
is a geodesic for the discriminant and oscillation norms
introduced in \cite{CS}.
The definition of the discriminant norm $\nu_{\dis}$
and of the oscillation pseudonorm $\nu_{\osc}$
on the universal cover of the identity component
of the contactomorphism group
of a closed contact manifold $(M, \xi)$
are recalled in \autoref{section: Reeb geodesic} below,
as well as the definition of the discriminant and oscillation lengths
of contact isotopies.
Recall also from \cite[Proposition 3.2]{CS}
that the oscillation pseudonorm
is non-degenerate if and only if $(M, \xi)$ is orderable;
in particular,
it is thus a norm for lens spaces.
As in \cite{PA},
we say that a contact isotopy
of a closed orderable contact manifold
is a geodesic for the discriminant or for the oscillation norm
if its discriminant or oscillation length
is equal to the discriminant or oscillation norm
of the element of the universal cover
it represents.
In other words,
a contact isotopy is a geodesic
for the discriminant or oscillation norm
if it minimizes the discriminant or oscillation length
in its homotopy class with fixed endpoints.
In \cite{PA} it is proved that certain contact isotopies
of $(\mathbb{R}^{2n} \times \mathbb{S}^1, \xi_0)$
are geodesics for the discriminant and oscillation norms.
We obtain here a similar result for lens spaces with equal weights,
answering a question in \cite{CS}.
In \cite{CS} and \cite{GKPS} respectively
it is proved that the discriminant and oscillation norms
on real projective spaces and on general lens spaces
are unbounded,
by showing that the classes in the universal cover
represented by higher iterations of the Reeb flow
have bigger and bigger discriminant and oscillation norms.
More precisely,
it is proved in \cite{GKPS} that for every $N$
the discriminant and oscillation norms
on $\widetilde{\Cont_0} \big(L_k^{2n-1} (\underline{w}), \xi_0 \big)$
of $\widetilde{r_{6 \pi N}}$
and $\widetilde{r_{20 \pi N}}$ respectively
are at least equal to $N + 1$.
Since the discriminant length of $\{ r_{6 \pi N t} \}_{t \in [0, 1]}$ is $3Nk + 1$
and the oscillation length of $\{ r_{20 \pi N t} \}_{t \in [0, 1]}$ is $10 N k + 1$,
the results in \cite{GKPS}
(as well as the previous ones in \cite{CS})
left open the question
of whether there exist contact isotopies
in the same homotopy class with fixed endpoints
as $\{ r_{6 \pi N t} \}_{t \in [0, 1]}$ or $\{ r_{20 \pi N t} \}_{t \in [0, 1]}$
having shorter discriminant or oscillation lengths.
Note that this is what happens for the sphere $(\mathbb{S}^{2n-1}, \bar{\xi}_0)$:
the $N$-th iteration $\{ r_{2 \pi N t} \}_{t \in [0, 1]}$
of the Reeb flow $\{ r_{2 \pi t} \}_{t \in [0, 1]}$ of $\bar{\alpha}_0$
has discriminant and oscillation length $N + 1$,
but by \cite[Proposition 4.3]{CS}
the discriminant norm and the oscillation pseudonorm
of $[\{ r_{2 \pi N t} \}_{t \in [0, 1]}]$ are smaller or equal than $4$;
in other words,
there exist contact isotopies of $(\mathbb{S}^{2n-1}, \bar{\xi}_0)$
in the same homotopy class with fixed endpoints
as certain iterations the Reeb flow
having strictly shorter discriminant and oscillation length.
As an application of \autoref{theorem: main},
in Section \ref{section: Reeb geodesic}
we show that for lens spaces with equal weights this is not possible.
More precisely,
we show that Theorem \ref{theorem: main}
implies the following result.

\begin{cor}[Non-shortening of the standard Reeb flow]\label{corollary: Reeb geodesic}
For every real number $T$,
the Reeb flow $\{ r_{Tt} \}_{t \in [0, 1]}$
of the standard contact form $\alpha_0$
on a lens space of the form $L_k^{2n-1} (w, \cdots, w)$
is a geodesic for the discriminant and oscillation norms.
In particular,
$$
\nu_{\dis} (\widetilde{r_T}) = \nu_{\osc} (\widetilde{r_T})
= \left\lfloor \frac{k}{2\pi} \, T \right\rfloor + 1 \,.
$$
\end{cor}
For general lens spaces $L_k^{2n-1} (\underline{w})$
we also obtain lower bounds
for the discriminant and oscillation norms of the standard Reeb flow
that are sharper than those in \cite{CS, GKPS},
however in general these bounds are not sharp enough
to prove that the Reeb flow is a geodesic
(see \autoref{remark: general lens spaces}).

Using the spectral selectors of Theorem \ref{theorem: main}
we also define a stably unbounded conjugation invariant pseudonorm
on $\widetilde{\Cont_0} \big( L_k^{2n-1} (\underline{w}), \xi_0 \big)$.
More precisely,
posing $c_- = c_{-2n+1}$ and $c_+ = c_0$
we prove the following result.

\begin{cor}[Spectral pseudonorm] \label{corollary: word norms}
The map $\nu: \widetilde{\Cont_0} \big(L_k^{2n-1} (\underline{w}), \xi_0 \big) \rightarrow T_{\underline{w}} \cdot \mathbb{Z}$
defined by
$$
\nu (\widetilde{\phi})
= \max \left\{\, \left\lceil c_+ (\widetilde{\phi}) \right\rceil_{T_{\underline{w}}} \,,\,
-  \left\lfloor c_- (\widetilde{\phi}) \right\rfloor_{T_{\underline{w}}} \,\right\}
$$
is a stably unbounded conjugation invariant pseudonorm,
which is compatible with the partial order $\leq$
and satisfies
$\nu (\widetilde{\phi}) \leq T_{\underline{w}} \cdot \nu_{\osc} (\widetilde{\phi})$
for every $\widetilde{\phi} \in \widetilde{\Cont_0} \big( L_k^{2n-1} (\underline{w}), \xi_0 \big)$.
The induced pseudonorm $\nu_{\ast}$ on $\Cont_0 \big( L_k^{2n-1} (\underline{w}), \xi_0 \big)$
is non-degenerate and bounded.
\end{cor}

Finally,
we remark that the spectral selectors of \autoref{theorem: main}
can also be used as in \cite{AA}
to define a time function  on $\widetilde{\Cont_0} \big( L_k^{2n-1} (\underline{w}), \xi_0 \big)$,
i.e. a function
$$
\tau: \widetilde{\Cont_0} \big( L_k^{2n-1} (\underline{w}), \xi_0 \big) \rightarrow \mathbb{R}
$$
that is continuous with respect to the $\mathcal{C}^1$-topology
and satisfies $\tau (\widetilde{\phi}) < \tau (\widetilde{\psi})$
whenever $\widetilde{\phi} \leq \widetilde{\psi}$ with $\widetilde{\phi} \neq \widetilde{\psi}$.
Such function can be defined by
$$
\tau (\widetilde{\phi})
= \left( \sum_j \frac{1}{2^j \max(1, | c_0 (\widetilde{\psi}_j) |)}  \right)^{-1}
\; \sum_j \frac{c_0 (\widetilde{\phi} \cdot \widetilde{\psi}_j)}{2^j \max(1, | c_0 (\widetilde{\psi}_j) |)} \;,
$$
where $(\widetilde{\psi}_j)_{j \geq 1}$
is any dense sequence in $\widetilde{\Cont_0} \big( L_k^{2n-1} (\underline{w}), \xi_0 \big)$
with respect to the $\mathcal{C}^1$-topology.
As in \cite{AA},
the time function $\tau$ satisfies moreover
$\tau ( \widetilde{r_T} \cdot \widetilde{\phi})
= T + \tau (\widetilde{\phi})$
for all $T$ and $\widetilde{\phi}$.

The article is organized as follows.
In Section \ref{section: non-linear Maslov index}
we recall the definition of the non-linear Maslov index
and discuss the properties that are needed
for the construction of the spectral selectors.
In Section \ref{section: action selectors} we define the spectral selectors
and prove Theorem \ref{theorem: main}.
In Section \ref{section: Reeb geodesic}
we prove Corollary \ref{corollary: Reeb geodesic},
and in Section \ref{section: conjugation invariant norms}
we prove Corollary \ref{corollary: word norms}.

\subsection*{Acknowledgments} 

The third author thanks Alberto Abbondandolo
for the invitation to Bochum and kind hospitality
in the fall of 2022,
when part of this work has been done.
We all thank Alberto also for his support
and for inspiring discussions.
The version of this article that was first posted to \emph{arXiv}
contained a mistake concerning the period of the Reeb flow
on lens spaces with general weights
and affecting in particular \autoref{corollary: Reeb geodesic},
which was stated in more generality
than what we can actually prove.
We are grateful to Miguel Abreu
for noticing this mistake
and for very useful discussions.

The first author was funded by the postdoctoral fellowships
of the \emph{Fondation Sciences Math\'ematiques de Paris}
and of the \emph{Interdisciplinary Thematic Institute IRMIA++},
as part of the \emph{ITI} 2021-2028 program of the University of Strasbourg, CNRS and Inserm, supported by \emph{IdEx Unistra (ANR-10-IDEX-0002)}
and by \emph{SFRI-STRAT’US project (ANR-20-SFRI-0012)}
under the framework of the \emph{French Investments for the Future Program}. 
The second author is partially supported
by the \emph{Deutsche Forschungsgemeinschaft}
under the \emph{Collaborative Research Center SFB/TRR 191 - 281071066 (Symplectic Structures in Geometry, Algebra and Dynamics)}.
The third author is partially supported
by the ANR project \emph{COSY (ANR-21-CE40-0002)}.


\section{The non-linear Maslov index}\label{section: non-linear Maslov index}

In this section we recall the definition of the non-linear Maslov index
$$
\mu: \widetilde{\Cont_0} \big( L_k^{2n-1} (\underline{w}), \xi_0 \big) \rightarrow \mathbb{Z}
$$
following the presentation in \cite{GKPS},
to which we refer for more details.
We also discuss the properties of the non-linear Maslov index
that are needed for the construction of the spectral selectors.
Several of these properties do not appear in \cite{GKPS},
and so we include detailed proofs.

Since the weights $\underline{w}$
do not play a particular role in the discussion,
in this section we denote $L_k^{2n-1} (\underline{w})$
simply by $L_k^{2n-1}$.
As in \cite{GKPS},
we first define the non-linear Maslov index
assuming  that $k$ is prime
and then obtain the general case
(\autoref{proposition: nonlinear Maslov index general k})
by pullback.
Assume thus for now that $k$ is prime.

The construction of the non-linear Maslov index
is based on generating functions.
Recall that a function $F: E \rightarrow \mathbb{R}$
defined on the total space of a fibre bundle $p: E \rightarrow B$
is said to be a generating function
if the differential $dF: E \rightarrow T^{\ast}E$
is transverse to the fibre conormal bundle $N_E^{\ast}$,
the space of points $(e, \eta)$ of $T^{\ast}E$
such that $\eta$ vanishes on the kernel of $dp (e)$.
Then the set $\Sigma_F = (dF)^{-1} (N_E^{\ast})$ of fibre critical points of $F$
is a submanifold of $E$,
and the map
$$
i_F: \Sigma_F \rightarrow T^{\ast}B \,,\; e \mapsto \big( p(e) , v^{\ast}(e) \big)
$$
defined by posing $v^{\ast}(e) (X) = dF (\widehat{X})$
for $X \in T_{p(e)}B$,
where $\widehat{X}$ is any vector in $T_eE$
with $dp (e) (\widehat{X}) = X$,
is a Lagrangian immersion
with respect to the canonical symplectic form $\omega_{\can}$ on $T^{\ast}B$.
If $i_F$ is an embedding then $F$ is said to be a generating function
of the Lagrangian submanifold $i_F (\Sigma_F)$ of $(T^{\ast}B, \omega_{\can})$.
A function $F$ is said to be a generating function
of a symplectomorphism $\Phi$ of $(\mathbb{R}^{2n}, \omega_0)$
if it is a generating function of the Lagrangian submanifold
of $(T^{\ast}\mathbb{R}^{2n}, \omega_{\can})$
that is the image of the graph of $\Phi$
by the symplectomorphism
$\tau: \overline{\mathbb{R}^{2n}} \times \mathbb{R}^{2n} \rightarrow T^{\ast}\mathbb{R}^{2n}$
defined by
$$
\tau (x, y, X, Y) = \Big( \frac{x + X}{2}, \frac{y + Y}{2}, Y - y, x - X \Big) \,.
$$

Any contact isotopy $\{\phi_t\}_{t \in [0, 1]}$ of $(L_k^{2n-1}, \xi_0)$
starting at the identity
can be uniquely lifted to a $\mathbb{Z}_k$-equivariant
contact isotopy $\{\bar{\phi}_t\}_{t \in [0, 1]}$ of $(\mathbb{S}^{2n-1}, \bar{\xi}_0)$
starting at the identity,
which in turn can be uniquely extended
to a conical Hamiltonian isotopy
$\{\Phi_t\}_{t \in [0, 1]}$ of $(\mathbb{R}^{2n}, \omega_0)$,
i.e.\ a 1-parameter family of $(\mathbb{Z}_k \times \mathbb{R}_{>0})$-equivariant
homeomorphisms of $\mathbb{R}^{2n}$
that is a Hamiltonian isotopy on $\mathbb{R}^{2n} \smallsetminus \{0\}$.
If $M$ is a multiple of $n$ then
we say that a function $F: \mathbb{R}^{2M} \rightarrow \mathbb{R}$
is conical if it is $\mathcal{C}^1$ with Lipschitz differential,
homogeneous of degree 2 with respect to the radial action of $\mathbb{R}_{>0}$
on $\mathbb{R}^{2M}$,
and invariant by the diagonal action of $\mathbb{Z}_k$ on $\mathbb{R}^{2M}$.
We say that a conical function $F: E \rightarrow \mathbb{R}$
defined on the total space of a trivial vector bundle
$E = \mathbb{R}^{2n} \times \mathbb{R}^{2nN} \rightarrow \mathbb{R}^{2n}$
is a conical generating functions
of a conical symplectomorphism $\Phi$ of $(\mathbb{R}^{2n}, \omega_0)$,
i.e. a $(\mathbb{Z}_k \times \mathbb{R}_{>0})$-equivariant homeomorphism of $\mathbb{R}^{2n}$
that is a symplectomorphism on $\mathbb{R}^{2n} \smallsetminus \{0\}$,
if it is smooth near its fibre critical points other than the origin,
$dF: E \rightarrow T^{\ast}E$ is transverse to the fibre conormal bundle $N_E^{\ast}$
except possibly at the origin,
and $i_F$ is a homeomorphism between $\Sigma_F$
and the image of the graph of $\Phi$ by $\tau$.
We say that $F_t: \mathbb{R}^{2n} \times \mathbb{R}^{2nN} \rightarrow \mathbb{R}$, $t \in [0, 1]$,
is a family of conical generating functions
for a contact isotopy $\{\phi_t\}_{t \in [0, 1]}$ of $(L_k^{2n-1}, \xi_0)$
starting at the identity
if for every $t$ the function $F_t$ is a conical generating function of $\Phi_t$,
where $\{\Phi_t\}$ denotes the conical Hamiltonian isotopy of $(\mathbb{R}^{2n}, \omega_0)$
lifting $\{\phi_t\}$,
and the map $(e, t) \mapsto F_t (e)$
is $\mathcal{C}^1$ with locally Lipschitz differential
and smooth near $(e, t)$
whenever $e$ is a fibre critical point of $F_t$
other than the origin.
We say that $F_t$, $t \in [0, 1]$,
is a based family of conical generating functions for $\{\phi_t\}_{t \in [0, 1]}$
if moreover $F_0$ is equivalent to the zero function on $\mathbb{R}^{2n}$,
where we consider on the set of conical generating functions
the smallest equivalence relation
under which two such functions are equivalent
if they differ by a stabilization
(i.e.\ replacing a conical generating function
$F:  \mathbb{R}^{2n} \times \mathbb{R}^{2nN} \rightarrow \mathbb{R}$
by $F \oplus Q:  \mathbb{R}^{2n} \times \mathbb{R}^{2nN} \times \mathbb{R}^{2nN'} \rightarrow \mathbb{R}$
for a non-degenerate $\mathbb{Z}_k$-invariant quadratic form
$Q$ on $\mathbb{R}^{2nN'}$)
or by a fibre preserving conical homeomorphism
(i.e.\ a $(\mathbb{Z}_k \times \mathbb{R}_{>0})$-equivariant homeomorphism
of $\mathbb{R}^{2n} \times \mathbb{R}^{2nN}$
that takes each fibre $\{z\} \times \mathbb{R}^{2nN}$ to itself)
that restricts to a diffeomorphism
between neighborhoods of fibre critical points other than the origin.
It is proved in \cite[Proposition 2.14]{GKPS}
that any contact isotopy $\{\phi_t\}_{t \in [0, 1]}$ of $(L_k^{2n-1}, \xi_0)$
starting at the identity
has a based family $F_t:  \mathbb{R}^{2n} \times \mathbb{R}^{2nN} \rightarrow \mathbb{R}$
of conical generating functions.

A conical function $F: \mathbb{R}^{2M} \rightarrow \mathbb{R}$
induces uniquely a function $f: L_k^{2M - 1} \rightarrow \mathbb{R}$,
which is $\mathcal{C}^1$ with Lipschitz differential.
All the critical points of $F$ have critical value zero
and come in $( \mathbb{Z}_k \times \mathbb{R}_{>0} )$-families;
moreover,
there is a 1--1 correspondence
between the $( \mathbb{Z}_k \times \mathbb{R}_{>0} )$-families
of critical points of $F$
and the critical points of critical value zero of $f$.
If $F$ is a conical generating function
of a conical symplectomorphism $\Phi$
whose restriction $\bar{\phi}$ to $\mathbb{S}^{2n-1}$
projects to a contactomorphism $\phi$ of $(L_k^{2n-1}, \xi_0)$
then there is a 1--1 correspondence
between the critical points of critical value zero of $f$
and the discriminant points of $\phi$
that lift to discriminant points of $\bar{\phi}$.
In order to detect discriminant points
of contactomorphisms of $(L_k^{2n-1}, \xi_0)$
we thus study the topology of the sublevel set at zero
of the functions
on (possibly higher dimensional) lens spaces
induced by the corresponding conical generating functions.
The topological invariant that we use for this
is the cohomological index
for subsets of lens spaces:
for a subset $A$ of $L_k^{2M-1}$
such index, which we denote by $\ind (A)$,
is the dimension over $\mathbb{Z}_k$
of the image of the map
$\check{H}^{\ast} ( L_k^{2M-1}; \mathbb{Z}_k ) \rightarrow \check{H}^{\ast} ( A; \mathbb{Z}_k )$
on \v{C}ech cohomology induced by the inclusion
$A \hookrightarrow L_k^{2M-1}$.
For any conical function $F: \mathbb{R}^{2M} \rightarrow \mathbb{R}$
we thus define
$$
\ind (F) = \ind ( \{ f \leq 0 \} ) \,.
$$

The following property is proved
in \cite[Corollary 3.15]{GKPS}
(cf \cite[Proposition 3.14, Proposition 3.9 (v) and Remark 3.11]{GKPS}
for the case $k = 2$).

\begin{lemma}\label{lemma: direct sum cohomological index}
For any two conical functions $F$ and $G$
we have
$$
\lvert \; \ind (F \oplus G) - \ind (F) - \ind (G) \; \lvert \leq 1 \,,
$$
and
$$
\ind (F \oplus G) = \ind (F) + \ind (G)
$$
if either $k = 2$ or $\ind (F)$ is even or $\ind (G)$ is even.
\end{lemma}

For a contact isotopy $\{\phi_t\}_{t \in [0, 1]}$ of $(L_k^{2n-1}, \xi_0)$
we define
$$
\mu (\{\phi_t\}_{t \in [0, 1]}) = \ind ( F_0 ) - \ind( F_1 ) \,,
$$
where $F_t$, $t \in [0, 1]$, is any based family
of conical generating functions for $\{\phi_t\}_{t \in [0, 1]}$.
It is proved in \cite[Proposition 2.20]{GKPS}
that any two based families
of conical generating functions for $\{\phi_t\}_{t \in [0, 1]}$
are equivalent,
where we consider on the set of based families of conical generating functions
the smallest equivalence relation
under which two such families are equivalent
if they differ by a stabilization
(i.e.\ replacing $F_t:  \mathbb{R}^{2n} \times \mathbb{R}^{2nN} \rightarrow \mathbb{R}$
by $F_t \oplus Q:  \mathbb{R}^{2n} \times \mathbb{R}^{2nN} \times \mathbb{R}^{2nN'} \rightarrow \mathbb{R}$
for a non-degenerate $\mathbb{Z}_k$-invariant quadratic form
$Q$ on $\mathbb{R}^{2nN'}$)
or by a 1-parameter family of fibre preserving conical homeomorphism
that restrict to diffeomorphisms
between neighborhoods of fibre critical points other than the origin.
Since, for $k > 2$, $\ind (Q)$ is even
for every $\mathbb{Z}_k$-invariant quadratic form $Q$
(\cite[Remark 3.13]{GKPS}),
it thus follows from \autoref{lemma: direct sum cohomological index}
that $\mu (\{\phi_t\}_{t \in [0, 1]})$ is well-defined,
i.e.\ it does not depend on the choice
of a based family of conical generating functions.
Moreover,
it is proved in \cite{GKPS}
(as a consequence of \cite[Proposition 2.21]{GKPS})
that $\mu$ descends to a map
$$
\mu: \widetilde{\Cont_0} (L_k^{2n-1}, \xi_0) \rightarrow \mathbb{Z} \,.
$$

\begin{ex}\label{example: identity etc}
By definition we have $\mu (\widetilde{\id}) = 0$.
By \cite[Example 4.1]{GKPS},
if $\widetilde{\phi}$ is small enough in the $\mathcal{C}^1$-topology
then $0 \leq \mu (\widetilde{\phi}) \leq 2n$,
and if moreover $\widetilde{\phi}$ is positive
then $\mu (\widetilde{\phi}) = 2n$.
In particular,
for $\epsilon > 0$ small enough
we have $\mu (\widetilde{r_{\epsilon}}) = 2n$.
By \cite[Example 4.13]{GKPS},
for every integer $m$
we have $\mu (\widetilde{r_{2 \pi m}}) = 2n m$.
\end{ex}

It is proved in \cite[Theorem 1.4 (i) and Remark 1.7]{GKPS}
that for any two elements $\widetilde{\phi}$ and $\widetilde{\psi}$
of $\widetilde{\Cont_0} (L_k^{2n-1}, \xi_0)$
we have
\begin{equation}\label{equation: quasimorphism}
\lvert\; \mu ( \widetilde{\phi} \cdot \widetilde{\psi} ) - \mu (\widetilde{\phi}) - \mu (\widetilde{\psi}) \;\rvert \leq 2n + 1 \,,
\end{equation}
and
\begin{equation}\label{equation: quasimorphism projective space}
\lvert\; \mu ( \widetilde{\phi} \cdot \widetilde{\psi} ) - \mu (\widetilde{\phi}) - \mu (\widetilde{\psi}) \;\rvert \leq 2n
\end{equation}
if $k = 2$;
in particular,
$\mu: \widetilde{\Cont_0} ( L_k^{2n-1}, \xi_0) \rightarrow \mathbb{Z}$
is a quasimorphism.
The non-linear Maslov index also satisfies the following triangle inequality,
which is not proved in \cite{GKPS}.

\begin{prop}\label{proposition: triangle inequality}
For any two elements $\widetilde{\phi}$ and $\widetilde{\psi}$
of $\widetilde{\Cont_0} (L_k^{2n-1}, \xi_0)$
we have
$$
\mu ( \widetilde{\phi} \cdot \widetilde{\psi} )
\leq \mu ( \widetilde{\phi} ) + \mu ( \widetilde{\psi} ) +1 \,,
$$
and
$$
\mu ( \widetilde{\phi} \cdot \widetilde{\psi} ) \leq \mu ( \widetilde{\phi} ) + \mu ( \widetilde{\psi} )
$$
if either $k = 2$
or $\mu ( \widetilde{\phi} )$ is even or $\mu ( \widetilde{\psi} )$ is even.
\end{prop}

Before proving \autoref{proposition: triangle inequality},
recall (\cite[Proposition 2.10]{GKPS})
that if $F \colon \mathbb{R}^{2n} \times \mathbb{R}^{2nN_1} \to \mathbb{R}$ 
and $G \colon \mathbb{R}^{2n} \times \mathbb{R}^{2nN_2} \to \mathbb{R}$
are conical generating functions for conical symplectomorphisms 
$\Phi$ and $\Psi$ respectively,
then the function  $F \, \sharp \, G \colon
\mathbb{R}^{2n} \times (\mathbb{R}^{2n} \times \mathbb{R}^{2n} \times \mathbb{R}^{2nN_1} \times \mathbb{R}^{2nN_2})
\rightarrow \mathbb{R}$
defined by
$$
F \, \sharp \, G \, (q; \zeta_1, \zeta_2, \nu_1,\nu_2) 
= F (\zeta_1, \nu_1) + G (\zeta_2, \nu_2) - 2 \left< \zeta_2 - q, i (\zeta_1- q) \right>
$$
is a conical generating function of $\Psi \circ \Phi$,
and (\cite[Proposition 2.26]{GKPS})
there is a linear $(\mathbb{Z}_k \times \mathbb{R}_{>0})$-equivariant injection
$$
\iota \colon \mathbb{R}^{2n} \times \mathbb{R}^{2n} \times \mathbb{R}^{2nN_1} \times \mathbb{R}^{2nN_2}
\rightarrow \mathbb{R}^{2n} \times \mathbb{R}^{2n} \times \mathbb{R}^{2n} \times \mathbb{R}^{2nN_1} \times \mathbb{R}^{2nN_2}
$$
such that $(F \, \sharp \, G) \circ \, \iota = F \oplus G$.
Since the cohomological index is monotone,
i.e. $\ind (A) \leq \ind (B)$ if $A \subset B$
\cite[Proposition 3.9 (i)]{GKPS},
we deduce that
\begin{equation}\label{equation: inclusion in triangle inequality}
\ind (F \oplus G) \leq \ind (F \, \sharp \, G) \,.
\end{equation}
Recall also (\cite[Lemma 4.2]{GKPS})
that if $F$ and $G$ are equivalent to the zero function
then
$$
\ind (F \, \sharp \, G) = \ind (F) + \ind(G) \,.
$$

\begin{proof}[Proof of \autoref{proposition: triangle inequality}]
Let $F_t$ and $G_t$ be based families
of conical generating functions
for contact isotopies $\{\phi_t\}_{t \in [0, 1]}$ and $\{\psi_t\}_{t \in [0, 1]}$
representing $\widetilde{\phi}$ and $\widetilde{\psi}$ respectively.
Then, by \cite[Proposition 2.10 and Remark 2.14]{GKPS},
$G_t \,\sharp\, F_t$ is a based family of conical generating functions
for $\{ \phi_t \circ \psi_t \}_{t \in [0, 1]}$,
and thus
$$
\mu ( \widetilde{\phi} \cdot \widetilde{\psi} )
= \ind (G_0 \,\sharp\, F_0) - \ind (G_1 \,\sharp\, F_1) \,.
$$
Since $F_0$ and $G_0$ are equivalent to the zero function,
we have $\ind (G_0 \,\sharp\, F_0) = \ind (G_0) + \ind (F_0)$.
Moreover, by \eqref{equation: inclusion in triangle inequality} we have
$$
- \ind (G_1 \,\sharp\, F_1) \leq - \ind (G_1 \oplus F_1) \,.
$$
Using \autoref{lemma: direct sum cohomological index}
we thus deduce that
$$
\mu ( \widetilde{\phi} \cdot \widetilde{\psi} ) \leq 
\ind (F_0) + \ind (G_0) - \ind (F_1) - \ind (G_1) + 1
= \mu ( \widetilde{\phi} ) + \mu ( \widetilde{\psi} ) + 1 \,,
$$
and that
$$
\mu ( \widetilde{\phi} \cdot \widetilde{\psi} ) \leq  \mu ( \widetilde{\phi} ) + \mu ( \widetilde{\psi} )
$$
if either $k = 2$ or $\ind (F_1)$ is even or $\ind (G_1)$ is even,
hence
(since, for $k > 2$, $\ind (F_0)$ and $\ind (G_0)$ are even
by \cite[Remark 3.13]{GKPS})
if either $k = 2$ or $\mu ( \widetilde{\phi} )$ is even or $\mu ( \widetilde{\psi} )$ is even.
\end{proof}

It is proved in \cite{Giv90}
that
$$
\mu ( \widetilde{\phi} \, \cdot \, \widetilde{\psi} )
= \mu ( \widetilde{\phi}) + \mu ( \widetilde{\psi})
$$
if either $\mathcal{L} (\widetilde{\phi})$ or $\mathcal{L} (\widetilde{\psi})$
are in $\pi_1 \big(\Cont_0 (\mathbb{S}^{2n-1}, \bar{\xi}_0) \big)$.
For our applications we only need this property
in the case when one of the factors is the Reeb flow.

\begin{prop}\label{proposition: concatenation}
For every element $\widetilde{\phi}$ of $\widetilde{\Cont_0} (L_k^{2n-1}, \xi_0)$
and every integer $m$ we have
\begin{equation*}
\mu ( \widetilde{\phi} \, \cdot \, \widetilde{r_{2 \pi m}} )
= \mu (\widetilde{r_{2 \pi m}} \, \cdot \, \widetilde{\phi}) = 2n m + \mu (\widetilde{\phi}) \,.
\end{equation*}
\end{prop}

\begin{proof}
Since $\widetilde{\phi} \, \cdot \, \widetilde{r_{2 \pi m}}
= \widetilde{r_{2 \pi m}} \, \cdot \, \widetilde{\phi}$,
it is enough to prove that
$\mu ( \widetilde{\phi} \, \cdot \, \widetilde{r_{2 \pi m}} )
= 2n m + \mu (\widetilde{\phi}) $.
We represent $\widetilde{\phi} \,\cdot\, \widetilde{r_{2 \pi m}}$
by the concatenation
$$
\{\varphi_t\}_{t \in [0, 1]} = \{ r_{4 \pi m t} \}_{t \in [0,\frac{1}{2}]} \sqcup \{ \phi_{2t-1} \}_{t \in [\frac{1}{2}, 1]} \,,
$$
where $\{ \phi_t \}_{t \in [0,1]}$ is a contact isotopy representing $\widetilde{\phi}$,
and we consider a based family $F_t$, $t \in [0, 1]$,
of conical generating functions for $\{\varphi_t\}_{t \in [0, 1]}$
so that $F_t$, $t \in [0, \frac{1}{2}]$, is a family of quadratic forms
generating $\{ r_{4 \pi m t} \}_{t \in [0, \frac{1}{2}]}$
(cf.\ \cite[Proposition 4.9]{GKPS}).
Since the lift of $\{ r_{4 \pi m t} \}_{t \in [0, \frac{1}{2}]}$
to $(\mathbb{S}^{2n-1}, \bar{\xi}_0)$
is a loop,
by \cite[Lemma 4.10]{GKPS} the quadratic form $F_{\frac{1}{2}}$
is equivalent to the zero function,
and so
$$
\mu ( \{\phi_t\}_{t \in [0, 1]} ) =  \ind ( F_{\frac{1}{2}} ) - \ind (F_1) \,.
$$
We thus have
$$
\mu (\widetilde{\phi} \, \cdot \, \widetilde{r_{2 \pi m}})
= \ind (F_0) - \ind (F_1)
= \ind (F_0) - \ind ( F_{\frac{1}{2}} ) + \ind ( F_{\frac{1}{2}} ) - \ind (F_1)
$$
$$
= \mu (\widetilde{r_{2 \pi m}}) + \mu (\widetilde{\phi})
= 2nm + \mu (\{ \phi_t \}_{t \in [0, 1]}) \,,
$$
where the last equality follows from \autoref{example: identity etc}.
\end{proof}

In the next section we also need the following result.

\begin{prop}\label{proposition: upper semi-continuous}
For every $\widetilde{\phi} \in \widetilde{\Cont_0} (L_k^{2n-1}, \xi_0)$
and every $T_0 \in \mathbb{R}$
there exists $\epsilon > 0$
such that 
$\mu ( \widetilde{r_{-T}} \cdot \widetilde{\phi} )
= \mu ( \widetilde{r_{-T_0}} \cdot \widetilde{\phi} )$
for every $T \in [T_0, T_0 + \epsilon)$.
\end{prop}

Before proving \autoref{proposition: upper semi-continuous},
recall that the cohomological index is continuous
(\cite[Proposition 3.9 (ii)]{GKPS}):
every closed subset $A$ of $L^{2M-1}_k$
has a neighborhood $\mathcal{U}$ such that 
if $A \subset \mathcal{V} \subset \mathcal{U}$
then $\ind(\mathcal{V}) = \ind(A)$.

\begin{proof}[Proof of \autoref{proposition: upper semi-continuous}]
Let $\{ \phi_t \}_{t \in [0, 1]}$ be a contact isotopy
representing $\widetilde{\phi}$,
and let $F_t$, $t \in [0, 1]$,
be a based family of conical generating functions
for the concatenation
$$
\{ \phi_{2t} \}_{t \in [0, \frac{1}{2}]} \sqcup \{ r_{- (2t-1)(T_0 + \epsilon')} \}_{t \in [\frac{1}{2}, 1]}
$$
for some $\epsilon' > 0$.
Then for every $T \in [T_0, T_0 + \epsilon')$ we have
$$
\mu (\widetilde{r_{-T}} \cdot \widetilde{\phi})
= \ind \big( \{ f_0 \leq 0 \} \big)
- \ind \Big( \Big\{ f_{\frac{1}{2} \big(\frac{T}{T_0 + \epsilon'} + 1 \big)} \leq 0 \Big\} \Big) \,.
$$
By monotonicity of generating functions
(\cite[Proposition 2.23]{GKPS})
we can assume that $\frac{df_t}{dt} \leq 0$ for all $t \in [\frac{1}{2}, 1]$,
and so
$$
\Big\{ f_{\frac{1}{2} (\frac{T}{T_0 + \epsilon'} + 1)} \leq 0 \Big\} 
\subset \Big\{ f_{\frac{1}{2} (\frac{T'}{T_0 + \epsilon'} + 1)} \leq 0 \Big\} 
$$
for $T, T' \in [ T_0, T_0 + \epsilon' )$ with $T \leq T'$.
By continuity of the cohomological index,
there is a neighborhood $\mathcal{U}$
of $\big\{ f_{\frac{1}{2} \big( \frac{T_0}{T_0 + \epsilon'} + 1 \big)} \leq 0 \big\} $
such that if
$\big\{ f_{\frac{1}{2} \big(\frac{T_0}{T_0 + \epsilon'} + 1 \big)} \leq 0 \big\}
\subset \mathcal{V} \subset \mathcal{U}$
then
$$
\ind (\mathcal{V}) = 
\ind \Big( \Big\{ f_{\frac{1}{2} \big(\frac{T_0}{T_0 + \epsilon'} + 1 \big)} \leq 0 \Big\} \Big) \,.
$$
For every $T \in [ T_0, T_0 + \epsilon)$
with $\epsilon < \epsilon'$ small enough
we have
$$
\Big\{ f_{\frac{1}{2} \big( \frac{T_0}{T_0 + \epsilon'} + 1 \big)} \leq 0 \Big\}
\subset \Big\{ f_{\frac{1}{2} \big( \frac{T}{T_0 + \epsilon'} + 1 \big)} \leq 0 \Big\}
\subset \mathcal{U} \,,
$$
and so
$$
\ind \Big( \Big\{ f_{\frac{1}{2} \big(\frac{T}{T_0 + \epsilon'} + 1 \big)} \leq 0 \Big\} \Big)
= \ind \Big( \Big\{ f_{\frac{1}{2} \big(\frac{T_0}{T_0 + \epsilon'} + 1 \big)} \leq 0 \Big\} \Big) \,,
$$
i.e. $\mu (\widetilde{r_{-T}} \cdot \widetilde{\phi}) = \mu (\widetilde{r_{-T_0}} \cdot \widetilde{\phi})$
as we wanted.
\end{proof}

It is proved in \cite[Proposition 4.21]{GKPS}
that if $\{\phi_t\}_{t \in [0, 1)}$ is a non-negative
(respectively non-positive) contact isotopy
then $\mu ( [ \{ \phi_t \}_{t \in [0, 1)} ] ) \geq 0$
(respectively $\mu ( [ \{ \phi_t \}_{t \in [0, 1)} ] ) \leq 0$).
We actually have the following result.

\begin{prop}\label{proposition: monotonicity Maslov}
If $\widetilde{\phi} \leq \widetilde{\psi}$
then $\mu (\widetilde{\phi}) \leq \mu (\widetilde{\psi})$. 
\end{prop}

\begin{proof}
Assume that $\widetilde{\phi} \leq \widetilde{\psi}$.
Then $\widetilde{\psi}$ can be represented by the concatenation
$$
\{\psi_t\}_{t \in [0, 1]} = \{\phi_{2t}\}_{t \in [0, \frac{1}{2}]} \sqcup \{\chi_{2t-1}\}_{t \in [\frac{1}{2}, 1]}
$$
of a contact isotopy $\{\phi_t\}_{t \in [0, 1]}$ representing $\widetilde{\phi}$
and a non-negative contact isotopy $\{\chi_t\}_{t \in [0, 1]}$.
Let $F_t: \mathbb{R}^{2n} \times \mathbb{R}^{2nN} \rightarrow \mathbb{R}$
be a based family of generating functions
for $\{\psi_t\}_{t \in [0, 1]}$.
By monotonicity of generating function
(\cite[Proposition 2.23]{GKPS})
we can assume that $F_1 \geq F_{\frac{1}{2}}$and so,
by monotonicity of the cohomological index
(\cite[Proposition 3.9 (i)]{GKPS}),
$\ind (F_1) \leq \ind (F_{\frac{1}{2}})$.
We thus have
$$
\mu (\widetilde{\psi}) = \mu (\widetilde{\phi}) + \ind (F_{\frac{1}{2}}) - \ind (F_1)
\geq \mu (\widetilde{\phi}) \,.
$$
\end{proof}

\begin{rmk}\label{remark: semi-continuous}
It follows from \autoref{proposition: upper semi-continuous}
and \autoref{proposition: monotonicity Maslov}
that the map $T \mapsto \mu ( \widetilde{r_{-T}} \cdot \widetilde{\phi} )$
is lower semi-continuous,
i.e. $\{ T \in \mathbb{R} \,\lvert\,  \mu ( \widetilde{r_{-T}} \cdot \widetilde{\phi} ) \leq y \}$
is closed for every $y \in \mathbb{R}$.
\end{rmk}

If $F_t$ is a based family of conical generating functions
for a contact isotopy $\{\phi_t\}_{t \in [0, 1]}$ of $(L_k^{2n-1}, \xi_0)$
then for every $t$ there is a 1--1 correspondence
between the critical points of critical value zero of $f_t$
and the discriminant points of $\phi_t$
that lift to discriminant points of $\bar{\phi}_t$,
where $\{ \bar{\phi}_t \}_{t \in [0, 1]}$
is the lift of $\{\phi_t\}_{t \in [0, 1]}$ to $(\mathbb{S}^{2n-1}, \bar{\xi}_0)$.
Since the non-linear Maslov index $\mu ( \{\phi_t\}_{t \in [0, 1]} )$
counts, with multiplicity given by the change in the cohomological index
of the sublevel sets $\{ f_t \leq 0 \}$,
the critical points of $f_t$ with critical value zero
as $t$ varies in $[0, 1]$,
its value is related to the presence of discriminant points of $\bar{\phi}_t$
for $t \in [0, 1]$.
More precisely,
we have the following result
(\cite[Theorem 1.4 (iii)]{GKPS}).

\begin{prop}\label{proposition: relation with discriminant points}
Let $\{\phi_t\}_{t \in [0, 1]}$ be a contact isotopy of $(L_k^{2n-1}, \xi_0)$
starting at the identity,
and let $[t_0, t_1]$ be a subinterval of $[0, 1]$.
If $\mu ( \{ \phi_t \}_{t \in [0, t_0]} ) \neq \mu ( \{ \phi_t \}_{t \in [0, t_1]} )$
then there is $\underline{t} \in [t_0, t_1]$
such that $\bar{\phi}_{\underline{t}}$ has discriminant points.
If moreover\footnote{In \cite[Theorem 1.4 (iii)]{GKPS} it is assumed
that there is only one $\underline{t} \in [0, 1]$
such that $\bar{\phi}_{\underline{t}}$ has discriminant points.
However,
the proof works in the same way
also if we only assume that
$s \mapsto \mu ( \{ \phi_t \}_{t \in [0, s]} )$
is constant on $[t_0, \underline{t})$ and on $(\underline{t}, t_1]$.
Similarly for \autoref{proposition: relation with discriminant points stronger}
below.}
the map $s \mapsto \mu ( \{ \phi_t \}_{t \in [0, s]} )$
is constant on $[t_0, \underline{t})$ and on $(\underline{t}, t_1]$
and all the discriminant points of  $\bar{\phi}_{\underline{t}}$
are non-degenerate then
$$
\lvert \, \mu ( \{ \phi_t \}_{t \in [0, t_0]} ) -  \mu ( \{ \phi_t \}_{t \in [0, t_1]} ) \, \rvert \leq 1 \,.
$$
\end{prop}

In particular,
it follows from \autoref{example: identity etc},
\autoref{proposition: concatenation}
and the first statement of \autoref{proposition: relation with discriminant points}
that for every real number $T$ we have
\begin{equation}\label{equation: Reeb flow}
\mu ( \widetilde{r_{T}} ) = 2n \left\lceil \frac{T}{2 \pi} \right\rceil \,.
\end{equation}

We also have the following result.

\begin{prop}\label{proposition: relation with discriminant points stronger}
Let $\{\phi_t\}_{t \in [0, 1]}$ be a contact isotopy of $(L_k^{2n-1}, \xi_0)$
starting at the identity,
and let $[t_0, t_1]$ be a subinterval of $[0, 1]$.
Assume that there is $\underline{t} \in [t_0, t_1]$
such that $\bar{\phi}_{\underline{t}}$ has discriminant points,
and denote by $\Delta (\phi_{\underline{t}}) \subset L_k^{2n-1}$
the set of discriminant points of $\phi_{\underline{t}}$.
Assume also that the map
$s \mapsto \mu ( \{ \phi_t \}_{t \in [0, s]} )$
is constant on $[t_0, \underline{t})$ and on $(\underline{t}, t_1]$.
Then
$$
\lvert \, \mu ( \{ \phi_t \}_{t \in [0, t_0]} ) -  \mu ( \{ \phi_t \}_{t \in [0, t_1]} ) \, \rvert
\leq \ind \big( \Delta (\phi_{\underline{t}}) \big) + 1 \,,
$$
and
$$
\lvert \, \mu ( \{ \phi_t \}_{t \in [0, t_0]} ) -  \mu ( \{ \phi_t \}_{t \in [0, t_1]} ) \, \rvert
\leq \ind \big( \Delta (\phi_{\underline{t}}) \big)
$$
if either $k = 2$
or if $\{\phi_t\}_{t \in [0, 1]}$ is negative
and $\mu ( \{ \phi_t \}_{t \in [0, t_0]} )$ is even.
\end{prop}

Before proving \autoref{proposition: relation with discriminant points stronger},
recall that the cohomological index is subadditive
in the following sense (\cite[Proposition 3.9 (iv)]{GKPS}):
for any two closed subsets $A$ and $B$ of $L_k^{2M-1}$
we have
$$
\ind (A \cup B) \leq \ind (A) + \ind (B) + 1 \,,
$$
and
$$
\ind (A \cup B) \leq \ind (A) + \ind (B)
$$
if either $k = 2$ or $\ind (A)$ is even or $\ind (B)$ is even.

\begin{proof}[Proof of \autoref{proposition: relation with discriminant points stronger}]
The first inequality
and the second one in the case $k = 2$
are proved in \cite[Proposition 4.15]{GKPS}.
Suppose thus that $\{\phi_t\}_{t \in [0, 1]}$ is negative
and $\mu ( \{ \phi_t \}_{t \in [0, t_0]} )$ is even.
By monotonicity of generating functions
(\cite[Proposition 2.23]{GKPS}),
there is a based family of conical generating functions $F_t$
for $\{\phi_t\}_{t \in [0, 1]}$
with $\frac{df_t}{dt} \leq 0$.
The proof of \cite[Proposition 2.23]{GKPS}
actually shows that for every $t$
either $\frac{df_t}{dt} < 0$ or $\frac{df_t}{dt} = 0$.
We can thus assume that either 
$\{ f_{t_0} \leq 0 \} = \{ f_{\underline{t}} \leq 0 \}$
or $\{ f_{t_0} \leq 0 \}$ is included in the interior of $\{ f_{\underline{t}} \leq 0 \}$.
In the first case,
by continuity of the cohomological index
and since the map $s \mapsto \mu ( \{ \phi_t \}_{t \in [0, s]} )$
is constant on $(\underline{t}, t_1]$
we have $\mu ( \{ \phi_t \}_{t \in [0, t_0]} ) =  \mu ( \{ \phi_t \}_{t \in [0, t_1]} )$,
which implies the desired inequality
$$
\lvert \, \mu ( \{ \phi_t \}_{t \in [0, t_0]} ) -  \mu ( \{ \phi_t \}_{t \in [0, t_1]} ) \, \rvert
\leq \ind \big( \Delta (\phi_{\underline{t}}) \big) \,.
$$
In the second case,
take $\epsilon, \epsilon' > 0$ such that
$$
\{ f_{t_0} \leq 0 \} \subset \{ f_{\underline{t}} \leq - \epsilon \}  \subset \{ f_{\underline{t} - \epsilon'} \leq 0 \} \subset \{ f_{\underline{t}} \leq 0 \} \,.
$$
By monotonicity of the cohomological index
and since $s \mapsto \mu ( \{ \phi_t \}_{t \in [0, s]} )$
is constant on $[t_0, \underline{t})$
we then have
\begin{equation}\label{equation: in stronger relation 1}
\ind \big( \{ f_{\underline{t}} \leq - \epsilon \} \big) = \ind \big( \{ f_{t_0} \leq 0 \} \big) \,.
\end{equation}
As in the proof of \cite[Proposition 4.15]{GKPS} we have
\begin{equation}\label{equation: in stronger relation 2}
\lvert \, \mu ( \{ \phi_t \}_{t \in [0, t_0]} ) -  \mu ( \{ \phi_t \}_{t \in [0, t_1]} ) \, \rvert
\leq \ind \big( \{ f_{\underline{t}} \leq \epsilon \} \big) - \ind \big( \{ f_{\underline{t}} \leq - \epsilon \} \big)
\end{equation}
and
\begin{equation}\label{equation: in stronger relation 3}
\ind \big( \{ f_{\underline{t}} \leq \epsilon \} \big) \leq \ind \big( \{ f_{\underline{t}} \leq - \epsilon \} \cup \mathcal{U} \big) \,,
\end{equation}
where $\mathcal{U}$ is a neighborhood of the set $C$
of critical points of $f_{\underline{t}}$ of critical value zero
that has the same cohomological index as $C$.
Since
$$
\mu ( \{ \phi_t \}_{t \in [0, t_0]} ) = \ind \big( \{ f_0 \leq 0 \}) - \ind \big( \{ f_{t_0} \leq 0 \})
$$
is even (by assumption)
and $\ind \big( \{ f_0 \leq 0 \} \big)$ is even
(by \cite[Remark 3.13]{GKPS}),
using \eqref{equation: in stronger relation 1} we see that 
$\ind \big( \{ f_{\underline{t}} \leq - \epsilon \} \big)$ is even
and so, by subadditivity of the cohomological index,
$$
\ind \big( \{ f_{\underline{t}} \leq - \epsilon \} \cup \mathcal{U} \big)
\leq \ind \big( \{ f_{\underline{t}} \leq - \epsilon \} \big) + \ind (\mathcal{U})
= \ind \big( \{ f_{\underline{t}} \leq - \epsilon \} \big) + \ind (C) \,.
$$
Since $\ind (C) = \ind \big( \Delta (\phi_{\underline{t}}) \big)$
by \cite[Proposition 2.22]{GKPS},
using \eqref{equation: in stronger relation 2} and \eqref{equation: in stronger relation 3}
we thus conclude that
$$
\lvert \, \mu ( \{ \phi_t \}_{t \in [0, t_0]} ) -  \mu ( \{ \phi_t \}_{t \in [0, t_1]} ) \, \rvert
\leq \ind \big( \Delta (\phi_{\underline{t}}) \big) \,.
$$
\end{proof}

The following Poincar\'e duality property
for the cohomological index
is proved in \cite[Proposition 4.1.15]{Car13}
for subsets of complex projective space.
We adapt here the proof to the case of lens spaces.

\begin{lemma}\label{lemma: PD cohomological index}
Assume that zero is a regular value
of a function $f: L_k^{2M-1} \rightarrow \mathbb{R}$.
Then
$$
\ind \big( \{ f \leq 0 \} \big) + \ind \big( \{ f \geq 0 \} \big) = 2M \,.
$$
\end{lemma}

\begin{proof}
Let $A = \{ f \leq 0 \}$ and $B = \{ f \geq 0 \}$.
Since zero is a regular value of $f$,
$A$ and $B$ are smooth submanifolds with boundary
and thus deformation retract to sets $A'$ and $B'$
that are strictly included in $\{ f < 0 \}$ and $\{ f > 0 \}$ respectively.
By continuity of the cohomological index,
there are thus open subsets
$\mathcal{U}_A$ and $\mathcal{U}_B$ of $L_k^{2M-1}$
strictly included in $\{ f < 0 \}$ and $\{ f > 0 \}$ respectively
with $\ind (\mathcal{U}_A) = \ind (A)$
and $\ind (\mathcal{U}_B) = \ind (B)$.
Assume first that $\ind(A)$ is even.
Since $A \cup B = L_k^{2M-1}$,
by subadditivity of the cohomological index
we have
$$
\ind (A) + \ind (B) \geq 2M \,.
$$
Assume by contradiction that
\begin{equation}\label{equation: in proof PD}
\ind (\mathcal{U}_A) + \ind (\mathcal{U}_B)
= \ind (A) + \ind (B) \geq  2M + 1 \,.
\end{equation}
Let $\ind (A) = \ind (\mathcal{U}_A) = 2a$.
By definition of the cohomological index
(cf.\ \cite[Lemma 3.3]{GKPS})
and since \v{C}ech cohomology
agrees with singular cohomology on open sets,
the homomorphism
$H^{2a-1} (L_k^{2M-1}; \mathbb{Z}_k)
\rightarrow H^{2a-1} (\mathcal{U}_A; \mathbb{Z}_k)$
induced by the inclusion
$\mathcal{U}_A \hookrightarrow L_k^{2M-1}$
is injective.
Since $k$ is prime,
the coefficient ring $\mathbb{Z}_k$ is a field
and so cohomology is the dual of homology,
thus the homomorphism
$H_{2a - 1} (\mathcal{U}_A; \mathbb{Z}_k)
\rightarrow H_{2a-1} (L_k^{2M-1}; \mathbb{Z}_k)$
induced by the inclusion
$\mathcal{U}_A \hookrightarrow L_k^{2M-1}$
is surjective.
Consider the commutative square
\begin{equation}\label{dia}
\begin{gathered}
\xymatrix{
H_{2a - 1} \big(L^{2M-1}_k \smallsetminus \mathrm{int}(B); \mathbb{Z}_k \big) \ar[r]
& H_{2a - 1} (L^{2M-1}_k; \mathbb{Z}_k) \\
H^{2(M - a)} (L^{2M-1}_k, B; \mathbb{Z}_k) \ar[u]^-{\cong} \ar[r]^-{j_B}
& H^{2(M - a)}(L^{2M-1}_k; \mathbb{Z}_k) \ar[u]^-{\cong}}
\end{gathered}
\end{equation}
where the horizontal arrows are induced by the inclusions
and the vertical ones are Poincaré duality isomorphisms
(composed with excision
$H^{2(M - a)} (L^{2M-1}_k, B; \mathbb{Z}_k) \rightarrow
H^{2(M - a)} (L^{2M-1}_k \smallsetminus \mathrm{int}(B), \partial B; \mathbb{Z}_k)$
for the vertical arrow on the left hand side).
Since $\mathcal{U}_A \subset L^{2M-1}_k \smallsetminus \mathrm{int}(B)$,
surjectivity of the inclusion homomorphism
$H_{2a - 1} (\mathcal{U}_A; \mathbb{Z}_k)
\rightarrow H_{2a-1} (L_k^{2M-1}; \mathbb{Z}_k)$
implies that the homomorphism
on the top horizontal line of the diagram
is also surjective.
Thus $j_B$ is surjective,
and so there exists a class
$$
u \in H^{2(M - a)} (L^{2M-1}_k, B; \mathbb{Z}_k)
$$
such that $j_B (u) = \beta^{M - a}$,
where $\beta$ denotes a generator
of $H^2 (L^{2M-1}_k; \mathbb{Z}_k)$.
By \eqref{equation: in proof PD}
we have $\ind (\mathcal{U}_B) \geq 2M + 1 - 2a$,
and so by a similar argument
there exists a class
$$
v \in H^{2a - 1}(L^{2M-1}_k, A; \mathbb{Z}_k)
$$
such that $j_A (v) = \alpha \beta^{a - 1}$,
where
$$
j_A: H_{2a - 1} (L_k^{2M-1}, A; \mathbb{Z}_k)
\rightarrow H_{2a-1} (L_k^{2M-1}; \mathbb{Z}_k)
$$
is the homomorphism induced by the inclusion
and where $\alpha$ denotes a generator
of $H^1 (L^{2M-1}_k; \mathbb{Z}_k)$.
We then obtain a contradiction:
on the one hand 
$v \cup u \in H^{2M-1}(L^{2M-1}_k, A \cup B; \mathbb{Z}_k)$ is zero
since $A \cup B = L^{2M - 1}_k$,
on the other hand by naturality of the cup product
we have
$$
j_{A \cup B} (v \cup u) 
= j_A (v) \cup  j_B (u)
= \alpha \beta^{M - 1}
\neq 0 \,. 
$$
This finishes the proof in the case when $\ind(A)$ is even.
If $\ind(B)$ is even a similar argument also gives the result.
Suppose thus that $\ind(A)$ and $\ind(B)$ are odd.
Then $\ind (A) + \ind (B)$ is even,
and so it is enough to prove
$$
2M - 1 \leq \ind (A) + \ind (B) \leq 2M \,.
$$
The first inequality follows from subadditivity of the cohomological index.
For the second one,
suppose by contradiction that
$$
\ind (\mathcal{U}_A) + \ind (\mathcal{U}_B) = \ind (A) + \ind (B) \geq 2M + 1 \,,
$$
and let $\ind(A) = 2a - 1$.
By a similar argument as above,
there exist cohomology classes
$u$ in $H^{2M - 2a + 1} (L^{2M-1}_k, B; \mathbb{Z}_k)$
and $v$ in $H^{2a - 2}(L^{2M-1}_k, A; \mathbb{Z}_k)$
such that
$j_B (u) = \alpha \beta^{M - a}$
and $j_A (v) = \beta^{a - 1}$.
Since
$j_B (u) \cup  j_A (v) = \alpha \beta^{M - 1} \neq 0$,
this leads to the same contradiction as above.
\end{proof}

Applying \autoref{lemma: PD cohomological index}
we obtain the following result.

\begin{prop}\label{proposition: PD Maslov}
For every $\widetilde{\phi} \in \widetilde{\Cont_0} (L_k^{2n-1}, \xi_0)$
such that $\Pi \big(\mathcal{L}(\widetilde{\phi})\big)$
does not have discriminant points
we have
$$
\mu (\widetilde{\phi}) + \mu (\widetilde{\phi}^{-1}) = 2n \,.
$$
\end{prop}

\begin{proof}
Let $F_t: \mathbb{R}^{2n} \times \mathbb{R}^{2nN} \rightarrow \mathbb{R}$
be a based family of conical generating functions
for a contact isotopy $\{ \phi_t \}_{t \in [0, 1]}$ representing $\widetilde{\phi}$.
Then $-F_t$ is a based family of conical generating functions
for $\{ \phi_t^{-1} \}_{t \in [0, 1]}$,
which represent $\widetilde{\phi}^{-1}$.
Thus
$$
\mu (\widetilde{\phi}) + \mu (\widetilde{\phi}^{-1})
= \ind (F_0) -  \ind (F_1) + \ind(- F_0) - \ind(- F_1) \,.
$$
Since $F_0$ is equivalent to the zero function,
up to a fibre preserving conical homeomorphism
it is equal to a $\mathbb{Z}_k$-invariant quadratic form $Q_0$.
Since $F_0$ generates the identity,
by \cite[Proposition 2.22]{GKPS} the nullity of $Q_0$ is equal to $2n$.
Let $\iota$ be the dimension of the maximal subspace
on which $Q_0$ is negative definite.
Then
$$
\ind (F_0) + \ind (- F_0) = \ind (Q_0) + \ind (- Q_0)
= (2n + \iota) + \big( 2n + (2nN - \iota) \big)
= 4n + 2nN \,.
$$
On the other hand,
since $\Pi \big(\mathcal{L}(\widetilde{\phi})\big)$ has no discriminant points
we have that zero is a regular value of $F_1$
and so we can apply \autoref{lemma: PD cohomological index}
to obtain
$$
\ind (F_1) + \ind (-F_1) = 2n + 2nN \,.
$$
We conclude that
$\mu (\widetilde{\phi}) + \mu (\widetilde{\phi}^{-1}) = 2n$.
\end{proof}

So far we have assumed that $k$ is prime.
Suppose now that $k$ is not prime,
and let $k'$ be the smallest prime that divides $k$.
As in  \cite[Remark 1.4]{GKPS},
we define the non-linear Maslov index
$$
\mu: \widetilde{\Cont_0} (L_k^{2n-1}, \xi_0) \rightarrow \mathbb{Z}
$$
by pulling back
$\mu: \widetilde{\Cont_0} (L_{k'}^{2n-1}, \xi_0) \rightarrow \mathbb{Z}$
by the natural map
$\widetilde{\Cont_0} (L_k^{2n-1}, \xi_0) \rightarrow \widetilde{\Cont_0} (L_{k'}^{2n-1}, \xi_0)$.
This general non-linear Maslov index
then satisfies the following properties.

\begin{prop}[Non-linear Maslov index for general $k$]
\label{proposition: nonlinear Maslov index general k}
For any integer $k \geq 2$
and $n$-tuple $\underline{w} = (w_1, \cdots, w_n)$
of positive integers relatively prime to $k$,
the non-linear Maslov index
$$
\mu: \widetilde{\Cont_0} \big( L_k^{2n-1} (\underline{w}), \xi_0 \big) \rightarrow \mathbb{Z}
$$
satisfies the following properties:
\renewcommand{\theenumi}{\roman{enumi}}
\begin{enumerate}

\item \label{examples general Maslov} \emph{Identity and small elements}:
if $\widetilde{\phi}$ is small enough in the $\mathcal{C}^1$-topology
then $0 \leq \mu (\widetilde{\phi}) \leq 2n$,
and if moreover $\widetilde{\phi}$ is positive
then $\mu (\widetilde{\phi}) = 2n$.
Moreover,
$\mu (\widetilde{\id}) = 0$.

\item \label{quasimorphism general Maslov} \emph{Quasimorphism property}:
$$
\lvert\; \mu ( \widetilde{\phi} \cdot \widetilde{\psi} ) - \mu (\widetilde{\phi}) - \mu (\widetilde{\psi}) \;\rvert \leq 2n + 1 \,,
$$
and
$$
\lvert\; \mu ( \widetilde{\phi} \cdot \widetilde{\psi} ) - \mu (\widetilde{\phi}) - \mu (\widetilde{\psi}) \;\rvert \leq 2n
$$
if $k$ is even;
in particular,
$\mu$ is a quasimorphism.

\item \label{triangle inequality general Maslov} \emph{Triangle inequality}:
$$
\mu ( \widetilde{\phi} \cdot \widetilde{\psi} ) \leq \mu ( \widetilde{\phi} ) + \mu ( \widetilde{\psi} ) +1 \,,
$$
and
$$
\mu ( \widetilde{\phi} \cdot \widetilde{\psi} ) \leq \mu ( \widetilde{\phi} ) + \mu ( \widetilde{\psi} )
$$
if either $k$ is even
or $\mu ( \widetilde{\phi} )$ is even
or  $\mu ( \widetilde{\psi} )$ is even.

\item \label{monotonicity general Maslov} \emph{Monotonicity}:
if $\widetilde{\phi} \leq \widetilde{\psi}$
then $\mu (\widetilde{\phi}) \leq \mu (\widetilde{\psi})$.

\item \label{relation with translated points general Maslov} \emph{Relation with discriminant points}:
for any contact isotopy $\{\phi_t\}_{t \in [0, 1]}$ of $\big( L_k^{2n-1} (\underline{w}), \xi_0 \big)$
starting at the identity
and any subinterval $[t_0, t_1]$ of $[0, 1]$,
if $\mu ( \{ \phi_t \}_{t \in [0, t_0]} ) \neq \mu ( \{ \phi_t \}_{t \in [0, t_1]} )$
then there is $\underline{t} \in [t_0, t_1]$
such that $\bar{\phi}_{\underline{t}}$ has discriminant points.
Assume that the map
$s \mapsto \mu ( \{ \phi_t \}_{t \in [0, s]} )$
is constant on $[t_0, \underline{t})$ and on $(\underline{t}, t_1]$,
and denote by $\Delta_{k'} (\phi_{\underline{t}}) \subset L_{k'}^{2n-1} (\underline{w})$
the set of discriminant points at time $\underline{t}$
of the lift of $\{\phi_t\}$ to $L_{k'}^{2n-1} (\underline{w})$,
where $k'$ is the smallest prime dividing $k$.
Then
$$
\lvert \, \mu ( \{ \phi_t \}_{t \in [0, t_0]} ) -  \mu ( \{ \phi_t \}_{t \in [0, t_1]} ) \, \rvert
\leq \ind \big( \Delta_{k'} (\phi_{\underline{t}}) \big) + 1 \,,
$$
and
$$
\lvert \, \mu ( \{ \phi_t \}_{t \in [0, t_0]} ) -  \mu ( \{ \phi_t \}_{t \in [0, t_1]} ) \, \rvert
\leq \ind \big( \Delta_{k'} (\phi_{\underline{t}}) \big)
$$
if either $k$ is even
or if $\{\phi_t\}_{t \in [0, 1]}$ is negative
and $\mu ( \{ \phi_t \}_{t \in [0, t_0]} )$ is even.
If moreover all the discriminant points of $\bar{\phi}_{\underline{t}}$
are non-degenerate then
$$
\lvert \, \mu ( \{ \phi_t \}_{t \in [0, t_0]} ) -  \mu ( \{ \phi_t \}_{t \in [0, t_1]} ) \, \rvert \leq 1 \,.
$$

\item \label{Reeb flow general Maslov} \emph{Reeb flow}:
for every real number $T$ we have
$$
\mu ( \widetilde{r_{T}} ) = 2n \left\lceil \frac{T}{2 \pi} \right\rceil \,.
$$

\item \label{concatenation with the Reeb flow general Maslov}
\emph{Composition with the Reeb flow}:
for every integer $m$ we have
$$
\mu ( \widetilde{\phi} \,\cdot\, \widetilde{r_{2 \pi m}} )
= \mu (\widetilde{r_{2 \pi m}} \,\cdot\, \widetilde{\phi}) = 2n m + \mu (\widetilde{\phi}) \,.
$$

\item \label{semicontinuity} \emph{Lower semi-continuity}:
the map $T \mapsto \mu ( \widetilde{r_{-T}} \cdot \widetilde{\phi} )$
is lower semi-continuous.

\item \label{PD general Maslov} \emph{Poincar\'e duality}:
for every $\widetilde{\phi}$ such that
$\Pi \big(\mathcal{L}(\widetilde{\phi})\big)$ has no discriminant points
we have
$$
\mu (\widetilde{\phi}) + \mu (\widetilde{\phi}^{-1}) = 2n \,.
$$

\end{enumerate}
\end{prop}

\begin{proof}
In the case when $k$ is prime
all the properties have been discussed above
in \autoref{example: identity etc},
\eqref{equation: quasimorphism},
\eqref{equation: quasimorphism projective space},
\autoref{proposition: triangle inequality},
\autoref{proposition: monotonicity Maslov},
\autoref{proposition: relation with discriminant points},
\autoref{proposition: relation with discriminant points stronger},
\eqref{equation: Reeb flow},
\autoref{proposition: concatenation},
\autoref{remark: semi-continuous}
and \autoref{proposition: PD Maslov}.
If $k$ is not prime
and $k'$ is the smallest prime that divides $k$
then the properties of 
$\mu: \widetilde{\Cont_0} \big( L_{k'}^{2n-1} (\underline{w}), \xi_0 \big) \rightarrow \mathbb{Z}$
imply the corresponding properties for the pullback
$\mu: \widetilde{\Cont_0} \big( L_k^{2n-1} (\underline{w}), \xi_0 \big) \rightarrow \mathbb{Z}$
by the natural map
$\widetilde{\Cont_0} \big( L_k^{2n-1} (\underline{w}), \xi_0 \big)
\rightarrow \widetilde{\Cont_0} \big( L_{k'}^{2n-1} (\underline{w}), \xi_0 \big)$.
\end{proof}


\section{Spectral selectors}\label{section: action selectors}

For any $j \in \Z$ we define the $j$-th spectral selector
on $\widetilde{\Cont_0} \big( L_k^{2n-1} (\underline{w}), \xi_0 \big)$ by
\begin{equation*}
c_j (\widetilde{\phi})
= \inf \big\{\, T \in \R  \;\big\lvert\; \mu (\widetilde{r_{-T}} \cdot \widetilde{\phi}) \leq - j  \,\big\} \,.
\end{equation*}
By \autoref{proposition: nonlinear Maslov index general k} (\ref{monotonicity general Maslov}),
(\ref{quasimorphism general Maslov})
and (\ref{Reeb flow general Maslov}),
the function $T \mapsto \mu (\widetilde{r_{-T}} \cdot \widetilde{\phi})$
is non-increasing
and  tends to $\mp \infty$
as $T \to \pm \infty$,
thus $c_j (\widetilde{\phi})$ is a well-defined real number.
By \autoref{proposition: nonlinear Maslov index general k} (\ref{semicontinuity})
the infimum is in fact a minimum,
in particular
\begin{equation}\label{equation: min}
\left\lceil c_j \big(\widetilde{\phi}\big) \right\rceil_{T_{\underline{w}}}
= \min \left\{\, N \in T_{\underline{w}} \cdot \mathbb{Z}  \;\ \Big\lvert\; \mu (\widetilde{r_{-N}} \cdot \widetilde{\phi}) \leq - j  \,\right\} \,.
\end{equation}
It follows from the definition that the sequence $\{c_j\}$
is non-decreasing.
In the rest of this section we prove
the other properties listed in \autoref{theorem: main}.

We say that a contactomorphism of $\big( L_k^{2n-1} (\underline{w}), \xi_0 \big)$
is non-degenerate with respect to $\alpha_0$
if all its translated points with respect to $\alpha_0$
are non-degenerate for all their translations.
We then have the following lemma.

\begin{lemma}\label{lemma: non-degenerate}
The set of contactomorphisms of $\big( L_k^{2n-1} (\underline{w}), \xi_0 \big)$
contact isotopic to the identity
that are non-degenerate with respect to $\alpha_0$
is dense in $\Cont_0 \big( L_k^{2n-1} (\underline{w}), \xi_0 \big)$
for the $\mathcal{C}^1$-topology.
\end{lemma}

\begin{proof}
We first prove the following result.
Let $\Lambda_0$ be a closed Legendrian submanifold
of a contact manifold $\big(M, \xi = \ker (\alpha) \big)$,
and denote by $\mathcal{L}eg (\Lambda_0)$
its Legendrian isotopy class.
Then for any collection $N = \{N_l \,,\, l \in \mathbb{Z} \}$
of submanifolds of $M$
the set $\mathcal{L}eg_N (\Lambda_0)$
of elements of $\mathcal{L}eg (\Lambda_0)$
transverse to $N_l$ for all $l$
is dense in $\mathcal{L}eg (\Lambda_0)$
for the $\mathcal{C}^1$-topology.
Indeed,
let $\Lambda$ be an element of $\mathcal{L}eg (\Lambda_0)$
and let $\mathcal{U} (\Lambda) \subset M$
be a Weinstein neighborhood of $\Lambda$.
Fix a diffeomorphism $\Psi$
from $\mathcal{U} (\Lambda)$
to an open neighborhood $\mathcal{U} (j^10)$
of the zero section of $J^1\Lambda$
such that $\Psi (\Lambda) = j^10$
and $\Psi^{\ast} (dz - \lambda_{\can}) = \alpha$.
Denote for each $l$ by $N_l'$ the submanifold
$\Psi \big( N_l \cap \mathcal{U} (\Lambda) \big)$ of $\mathcal{U} (j^10)$,
and let $N' = \{N'_l \,,\, l \in \mathbb{Z} \}$.
Since the map
$j^1: \mathcal{C}^{\infty} (\Lambda) \rightarrow \mathcal{L}eg (j^10)$
that associates to a function its 1-jet
is a local homeomorphism
with respect to the $C^2$-topology on $\mathcal{C}^{\infty} (\Lambda)$
and the $\mathcal{C}^1$-topology on $\mathcal{L}eg (j^10)$
(cf.\ for instance \cite[Section 3]{Tsuboi3}),
for a sufficiently small $\mathcal{C}^2$-neighborhood $\mathcal{U}(0)$
of the zero function in $\mathcal{C}^{\infty} (\Lambda)$
the 1-jet of any $f \in \mathcal{U}(0)$ is in $\mathcal{U}(j^10)$
and the map $\mathcal{U}(0) \rightarrow \mathcal{L}eg (\Lambda_0)$
that sends $f$ to $\Psi^{-1} (j^1f)$
is a local homeomorphism.
By Thom's transversality theorem
(see for instance \cite[Corollary 4.10]{golubitsky2012stable}),
the subset $T^1_{N'}$ of $\mathcal{U} (0)$
consisting of functions with 1-jet transverse to $N_l'$ for all $l$
is dense in $\mathcal{U} (0)$
for the $\mathcal{C}^{\infty}$-topology,
hence also for the $\mathcal{C}^2$-topology.
Thus for any $\mathcal{C}^1$-neighborhood
$\mathcal{U}$ of $\Lambda$ in $\mathcal{L}eg (\Lambda_0)$
there exists $f \in T^1_{N'}$
such that $\Psi^{-1} (j^1f) \in \mathcal{U}$.
The Legendrian $\Psi^{-1} (j^1f)$
intersects each $N_l$ transversely,
and so belongs to $\mathcal{L}eg_N (\Lambda_0)$.
This shows that 
$\mathcal{L}eg_N (\Lambda_0)$
is $\mathcal{C}^1$-dense in $\mathcal{L}eg (\Lambda_0)$.

Using this,
we now prove that 
the set of contactomorphisms of $\big( L_k^{2n-1} (\underline{w}), \xi_0 \big)$
contact isotopic to the identity
that are non-degenerate with respect to $\alpha_0$
is $\mathcal{C}^1$-dense in $\Cont_0 \big( L_k^{2n-1} (\underline{w}), \xi_0 \big)$.
Consider the contact product of $\big( L_k^{2n-1} (\underline{w}), \xi_0 \big)$,
i.e.\ the product $L_k^{2n-1} (\underline{w}) \times L_k^{2n-1} (\underline{w}) \times \mathbb{R}$
endowed with the contact structure
given by the kernel of the contact form
$\pi_2^{\ast}\alpha_0 - e^{\theta} \pi_1^{\ast} \alpha_0$,
where $\pi_1$ and $\pi_2$ denote
the projections on the first and second factor respectively
and where $\theta$ is the coordinate in $\mathbb{R}$.
Denote by $\gr_{\alpha_0} (\phi)$ the graph
with respect to $\alpha_0$
of a contactomorphism $\phi$ of $\big( L_k^{2n-1} (\underline{w}), \xi_0 \big)$,
i.e.\ the Legendrian submanifold
$$
\gr_{\alpha_0} (\phi) = \big\{ \big(p, \phi(p), g(p)\big) \;\big\lvert\; p \in L_k^{2n-1} (\underline{w}) \big\}
$$
of the contact product,
where $g$ is the conformal factor of $\phi$ with respect to $\alpha_0$.
Denote by $\{R_t\}$ the Reeb flow
of $\pi_2^{\ast}\alpha_0 - e^{\theta} \pi_1^{\ast} \alpha_0$.
Let $0 = t_0 < \cdots < t_m = 2 \pi$
be a decomposition of the time interval $[0, 2 \pi]$
such that, for some $\epsilon > 0$,
$N_l := \bigcup_{t \in (t_l - \epsilon, t_{l + 1} + \epsilon)} \, R_t \big( \gr_{\alpha_0} (\id) \big)$
is a submanifold of $L_k^{2n-1} (\underline{w}) \times L_k^{2n-1} (\underline{w}) \times \mathbb{R}$
for every $l \in \{0, \cdots, m-1\}$.
A contactomorphism $\phi$ of $\big( L_k^{2n-1} (\underline{w}), \xi_0 \big)$
is non-degenerate with respect to $\alpha_0$
if and only if $\gr_{\alpha_0} (\phi)$
is transverse to $N_l$ for all $l$.
In other words,
using the notation of the first part of the proof,
$\phi$ is non-degenerate with respect to $\alpha_0$
if and only if $\gr_{\alpha_0} (\phi) \in \mathcal{L}eg_N (\gr_{\alpha_0} (\id))$,
where $N = \{ N_0, \cdots, N_{m-1} \}$.
Our result thus follows from the first part of the proof
and the fact that the map
$$
\gr_{\alpha_0}: \Cont_0 \big( L_k^{2n-1} (\underline{w}), \xi_0 \big)
\rightarrow \mathcal{L}eg (\gr_{\alpha_0}(\id))
$$
that associates to a contactomorphism its graph
is a local homeomorphism
with respect to the $\mathcal{C}^1$-topologies.
\end{proof}

We also remark the following fact.

\begin{lemma}\label{lemma: closed}
For any element $\widetilde{\phi}$ of $\widetilde{\Cont_0} \big( L_k^{2n-1} (\underline{w}), \xi_0 \big)$,
the sets $\mathcal{A} (\widetilde{\phi})$ and $\bar{\mathcal{A}} (\widetilde{\phi})$
are closed and nowhere dense in $\mathbb{R}$.
\end{lemma}

\begin{proof}
Let $F_t: \mathbb{R}^{2n} \times \mathbb{R}^{2nN} \rightarrow \mathbb{R}$,
$t \in [0, 1]$,
be a based family of conical generating functions
for the concatenation
$$
\{ \phi_{2t} \}_{t \in [0, \frac{1}{2} ]} \sqcup \{ r_{- 2\pi (2t - 1)} \}_{t \in [\frac{1}{2}, 1]} \,. 
$$
For every $t \in [\frac{1}{2}, 1]$,
the $\mathbb{Z}_k$-orbits of translated points of $\bar{\phi}_1$ of translation $2\pi (2t - 1)$
are in 1--1 correspondence with the critical points of critical value zero
of $f_t: L_k^{2n(N+1)-1} \rightarrow \mathbb{R}$.
Consider the function
$$
f: L_k^{2n(N+1)-1} \times \Big[ \frac{1}{2}, 1 \Big] \rightarrow \mathbb{R} \,,\;
(x, t) \mapsto f_t (x) \,.
$$
As in \cite[Section 5.2 and Lemma 4.4]{Theret-Rotation},
zero is a regular value of $f$
and so $f^{-1}(0)$ is a submanifold of $L_k^{2n(N+1)-1} \times [\frac{1}{2}, 1]$.
Let $p: f^{-1}(0) \rightarrow \big[ \frac{1}{2}, 1 \big]$
be the composition of the inclusion
$f^{-1}(0) \hookrightarrow L_k^{2n(N+1)-1} \times [\frac{1}{2}, 1]$
with the projection on the second factor.
Then $(x, t) \in f^{-1} (0)$ is a critical point of $p$
(of critical value $t$)
if and only if $x$ is a critical point of $f_t$
(of critical value zero).
Thus the $\mathbb{Z}_k$-orbits of translated points of $\bar{\phi}_1$
 of translation $2\pi (2t - 1)$
are in 1--1 correspondence with the critical points of $p$
of critical value $t$.
It follows that
$$
\bar{\mathcal{A}} (\widetilde{\phi}) =
2\pi \Big( 2 \, p \big( \Crit (p) \big) - 1 \Big) + 2 \pi \cdot \mathbb{Z} \,,
$$
and so $\bar{\mathcal{A}} (\widetilde{\phi})$ is closed and nowhere dense.
Since
$$
\bar{\mathcal{A}} (\widetilde{\phi}) + T_{\underline{w}} \cdot \mathbb{Z}
\subseteq \mathcal{A} (\widetilde{\phi})
\subseteq \bar{\mathcal{A}} (\widetilde{\phi}) + \frac{2\pi}{k} \cdot \mathbb{Z} \,,
$$
we deduce that $ \mathcal{A} (\widetilde{\phi})$
is also closed and nowhere dense.
\end{proof}

We now prove that the spectral selectors
satisfy the properties listed in \autoref{theorem: main}.

\subsection*{Spectrality}

We have to show that
$$
c_j (\widetilde{\phi}) \in \bar{\mathcal{A}} (\widetilde{\phi})
$$
for every $\widetilde{\phi}$
in $\widetilde{\Cont_0} \big( L_k^{2n-1} (\underline{w}), \xi_0 \big)$.
Suppose by contradiction that this is not the case.
Since $\bar{\mathcal{A}} (\widetilde{\phi})$ is closed
(by \autoref{lemma: closed}),
there is then $\epsilon > 0$ such that
$[ c_j (\widetilde{\phi}) - \epsilon,  c_j (\widetilde{\phi}) + \epsilon ]
\subset \mathbb{R} \smallsetminus \bar{\mathcal{A}} (\widetilde{\phi})$.
By the first statement
of \autoref{proposition: nonlinear Maslov index general k}
(\ref{relation with translated points general Maslov})
we thus have
$$
\mu \big( \widetilde{ r_{- ( c_j (\widetilde{\phi}) - \epsilon )} } \cdot \widetilde{\phi} \big)
= \mu \big( \widetilde{ r_{- ( c_j (\widetilde{\phi}) + \epsilon )} } \cdot \widetilde{\phi} \big) \,,
$$
but this contradicts the definition of $c_j (\widetilde{\phi})$.

\subsection*{Normalization}

It follows from
\autoref{proposition: nonlinear Maslov index general k}
(\ref{Reeb flow general Maslov})
that
$$
c_j (\widetilde{\id}) =
\begin{cases}
\cdots \\
- 4 \pi  & \text{ for } j = - 6n + 1, \cdots, - 4n \\
- 2 \pi  & \text{ for } j = - 4n + 1, \cdots, - 2n \\
0   & \text{ for } j = -2n + 1, \cdots, 0 \\
2 \pi   & \text{ for } j = 1, \cdots, 2n \\
4 \pi   & \text{ for } j = 2n + 1, \cdots, 4n \\
\cdots 
\end{cases}
$$
In particular,
$c_0 (\widetilde{\id}) = 0$.

\subsection*{Relation with translated points}

If all the translated points
of $\Pi \big( \mathcal{L} (\widetilde{\phi})\big)$
are non-degenerate,
it follows from the last statement of
\autoref{proposition: nonlinear Maslov index general k}
(\ref{relation with translated points general Maslov})
that the spectral selectors
$\{ c_j(\widetilde{\phi}) \,,\, j \in \mathbb{Z} \}$
are all distinct.
Suppose now that
$$
c_{j - 1} (\widetilde{\phi}) < c_{j} (\widetilde{\phi}) = c_{j + 1} (\widetilde{\phi}) = \cdots = c_{j + m} (\widetilde{\phi})
= T < c_{j + m + 1} (\widetilde{\phi})
$$
for some $j$ and $1 \leq m \leq 2n-1$.
Then for $\epsilon > 0$ small enough we have 
$$
\lvert\, \mu \big( \widetilde{r_{-(T - \epsilon)}} \cdot \widetilde{\phi} \big)
- \mu \big( \widetilde{r_{-(T + \epsilon)}} \cdot \widetilde{\phi} \big) \,\lvert
= m + 1
$$
and $\mu \big( \widetilde{r_{-(T - \epsilon)}} \cdot \widetilde{\phi} \big) = - j + 1$.
It thus follows from
\autoref{proposition: nonlinear Maslov index general k}
(\ref{relation with translated points general Maslov})
that if either $k$ is even or $j$ is odd or $m > 1$
then $\Pi \big( \mathcal{L} (\widetilde{\phi})\big)$
has infinitely many translated points
of translation $T$.
If moreover $k$ is prime then
\autoref{proposition: nonlinear Maslov index general k}
(\ref{relation with translated points general Maslov})
also implies that
the set of translated points of translation $T$ of $\Pi (\widetilde{\phi})$
has cohomological index greater or equal than $m$,
and greater or equal than $m + 1$
if either $k = 2$ or $j$ is odd.

\subsection*{Non-degeneracy}

Assume that
$$
c_{-2n + 1} (\widetilde{\phi}) = c_0 (\widetilde{\phi}) = 0 \,.
$$
Then for $\epsilon > 0$ small enough
we have
$$
\lvert\, \mu \big( \widetilde{r_{\epsilon}} \cdot \widetilde{\phi} \big)
- \mu \big( \widetilde{r_{- \epsilon}} \cdot \widetilde{\phi} \big) \,\lvert
= 2n
$$
and $\mu \big( \widetilde{r_{\epsilon}} \cdot \widetilde{\phi} \big) = 2n$.
By \autoref{proposition: nonlinear Maslov index general k}
(\ref{relation with translated points general Maslov})
we then conclude
that $\Pi (\widetilde{\phi})$ is the identity.
Since, by spectrality, $0 \in \bar{\mathcal{A}} (\widetilde{\phi})$,
we have in fact that $\Pi \big( \mathcal{L} (\widetilde{\phi})\big)$ is the identity.

\subsection*{Composition with the Reeb flow}

By definition of the spectral selectors,
for every $\widetilde{\phi}$ and every $T \in \mathbb{R}$
we have
\begin{align*}
c_j (\widetilde{r_T} \cdot \widetilde{\phi})
&= \inf \{\, \underline{T} \in \R  \;\ |\; \mu ( \widetilde{ r_{- \underline{T}} } \cdot \widetilde{r_T} \cdot \widetilde{\phi}) \leq - j  \,\} \\
&= \inf \{\, \underline{T} - T \in \R  \;\ |\; \mu (\widetilde{r_{- \underline{T}+T}} \cdot \widetilde{\phi}) \leq - j  \,\} + T \\
&= c_j (\widetilde{\phi}) + T \,.
\end{align*}

\subsection*{Periodicity}

Using the previous property
and \autoref{proposition: nonlinear Maslov index general k}
(\ref{concatenation with the Reeb flow general Maslov})
we have
$$
c_j (\widetilde{\phi}) + 2 \pi = c_j (\widetilde{r_{2 \pi}} \cdot \widetilde{\phi}) = c_{j + 2n} (\widetilde{\phi}) \,.
$$

\subsection*{Monotonicity}

Suppose that $\widetilde{\phi} \leq \widetilde{\psi}$.
Then $\widetilde{r_{-T}} \cdot \widetilde{\phi} \leq \widetilde{r_{-T}} \cdot \widetilde{\psi}$
and so,
by \autoref{proposition: nonlinear Maslov index general k}
(\ref{monotonicity general Maslov}),
$$
\mu (\widetilde{r_{-T}} \cdot \widetilde{\phi}) \leq \mu (\widetilde{r_{-T}} \cdot \widetilde{\psi}) \,.
$$
This implies that $c_j (\widetilde{\phi}) \leq c_j (\widetilde{\psi})$.

\subsection*{Continuity}

Suppose that $\widetilde{\phi} \cdot \widetilde{\psi}^{-1}$ is represented
by a contact isotopy with Hamiltonian function
$H_t: L_k^{2n-1} (\underline{w}) \rightarrow \mathbb{R}$
with respect to $\alpha_0$.
Let $m(t) = \min H_t$ and $M(t) = \max H_t$.
The flows of $m$ and $M$ are respectively
$\big\{ r_{\int^t_0 m} \big\}$ and $\big\{ r_{\int^t_0 M} \big\}$,
thus
\begin{equation*}
\widetilde{r_{\int^1_0 m}}
\leq \widetilde{\phi} \cdot \widetilde{\psi}^{-1}
\leq \widetilde{r_{\int^1_0 M}}  
\end{equation*}
and so
\begin{equation*}
\widetilde{r_{\int^1_0 m}} \cdot \widetilde{\psi} \leq \widetilde{\phi} \leq \widetilde{r_{\int^1_0 M}} \cdot \widetilde{\psi} \,.
\end{equation*}
By the composition with the Reeb flow property
and monotonicity we thus have
\begin{equation}\label{equation: continuity in section 3}
\int_0^1 \min H_t \, dt
\leq c_j (\widetilde{\phi}) - c_j (\widetilde{\psi})
\leq \int_0^1 \max H_t \, dt \,.
\end{equation}
We now show that \eqref{equation: continuity in section 3}
implies that each $c_j$ is continuous
with respect to the $\mathcal{C}^1$-topology.
Notice first that the Shelukhin--Hofer norm
$\nu_{\alpha}: \widetilde{\Cont_0} \big( L_k^{2n-1} (\underline{w}), \xi_0 \big) \rightarrow \mathbb{R}$,
which is defined by
$$
\nu_{\alpha} (\widetilde{\phi})
= \inf \int_0^1 \max \, \lvert H_t \lvert \, dt
$$
with the infimum taken over all contact Hamiltonian functions $H_t$
whose flow represents $\widetilde{\phi}$,
is continuous
with respect to the $\mathcal{C}^1$-topology.
This follows from the main result of \cite{AA2},
or can be seen directly as follows.
Since $\widetilde{\Cont_0} \big( L_k^{2n-1} (\underline{w}), \xi_0 \big)$
is a topological group,
it is enough to show that $\nu_{\alpha}$
is $\mathcal{C}^1$-continuous at the identity,
i.e.\ that for every $\epsilon > 0$
there is a $\mathcal{C}^1$-neighborhood $\widetilde{\mathcal{U}}$
of $\widetilde{\id}$ in $\widetilde{\Cont_0} \big(  L_k^{2n-1} (\underline{w}), \xi_0 \big)$
such that for every $\widetilde{\phi} \in \widetilde{\mathcal{U}}$
we have $\nu_{\alpha} (\widetilde{\phi}) < \epsilon$.
As in the proof of \autoref{lemma: non-degenerate}
we consider the product $L_k^{2n-1} (\underline{w}) \times L_k^{2n-1} (\underline{w}) \times \mathbb{R}$
endowed with the contact structure
given by the kernel of the contact form
$\pi_2^{\ast} \alpha_0 - e^{\theta} \pi_1^{\ast} \alpha_0$.
Applying the Weinstein theorem
we can find a neighborhood $\mathcal{U} ( \gr_{\alpha_0} (\id) )$
of $\gr_{\alpha_0} (\id)$,
a neighborhood $\mathcal{U} (j^10)$
of the zero section of $J^1 L_k^{2n-1} (\underline{w})$
of the form $\mathcal{U}(j^10) = \underline{\mathcal{U}}(j^10) \times (-\epsilon', \epsilon')$
with $\epsilon' < \epsilon$
and a diffeomorphism $\Psi$ from $\mathcal{U} ( \gr_{\alpha_0} (\id) )$ to $\mathcal{U} (j^10)$
with $\Psi ( \gr_{\alpha_0} (\id) ) = j^10$ and
$\Psi^{\ast} (dz - \lambda_{\can}) = \pi_2^{\ast} \alpha_0 - e^{\theta} \pi_1^{\ast} \alpha_0$.
Since the map $j^1: \mathcal{C}^{\infty} \big( L_k^{2n-1} (\underline{w}) \big)) \rightarrow \mathcal{L}eg (j^10)$
that associates to a function its 1-jet
is a local homeomorphism
with respect to the $\mathcal{C}^2$-topology on $\mathcal{C}^{\infty} \big( L_k^{2n-1} (\underline{w}) \big)$
and the $\mathcal{C}^1$-topology on $\mathcal{L}eg (j^10)$,
we can find a convex $\mathcal{C}^2$-neighborhood $\mathcal{U}(0)$
of the zero function in $\mathcal{C}^{\infty} \big(L_k^{2n-1} (\underline{w}) \big)$
such that $j^1f \in \mathcal{U} (j^10)$ for any $f \in \mathcal{U}(0)$,
and the map $\mathcal{U}(0) \mapsto \mathcal{L}eg ( \gr_{\alpha_0} (\id) )$
that sends $f$ to $\Psi^{-1} (j^1f)$ is a local homeomorphism.
Since the map $\gr_{\alpha_0}: \Cont_0 \big( L_k^{2n-1} (\underline{w}), \xi_0 \big)
\rightarrow \mathcal{L}eg \big( \gr_{\alpha_0}(\id)\big)$
that associates to a contactomorphism its graph
is a local homeomorphism
with respect to the $\mathcal{C}^1$-topologies,
we obtain a map $\mathcal{U}(0) \rightarrow  \Cont_0 \big( L_k^{2n-1} (\underline{w}), \xi_0 \big)$
that associates to a function $f$
a contactomorphism $\phi$ with $j^1f = \Psi (\gr_{\alpha_0}(\phi))$,
which is a homeomorphism on its image $\mathcal{U}$.
Since $\mathcal{U}(0)$ is convex,
$\mathcal{U}$ is simply connected;
let thus $\widetilde{\mathcal{U}}$ be the open neighborhood
of $\widetilde{\id}$ in $\widetilde{\Cont_0} \big( L_k^{2n-1} (\underline{w}), \xi_0 \big)$
that projects homeomorphically to $\mathcal{U}$.
Consider now any $\widetilde{\phi}$ in $\widetilde{\mathcal{U}}$.
Let $f: L_k^{2n-1} (\underline{w}) \rightarrow (- \epsilon, \epsilon)$
be the function in $\mathcal{U}(0)$
such that $j^1f = \Psi \big(\gr_{\alpha_0} \big(\Pi (\widetilde{\phi}) \big) \big)$.
Since $\mathcal{U}(0)$ is convex,
$f_t := tf$ is in $\mathcal{U}(0)$ for every $t \in [0, 1]$,
and so for every $t \in [0, 1]$ there is $\phi_t \in \mathcal{U}$
with $j^1f_t = \Psi \big( \gr_{\alpha_0} (\phi_t) \big)$.
Consider the two Legendrian isotopies
$j^1f_t$ and $\gr_{\alpha_0} (\phi_t)$,
with parametrisations given respectively by
$$
i_1: [0, 1] \times j^10 \rightarrow J^1 L_k^{2n-1} (\underline{w}) \,,\;
i_1 \big( t, (x, 0, 0) \big) = j^1 f_t (x)
$$
and
$$
i_2: [0, 1] \times \gr_{\alpha_0} (\id)
\rightarrow L_k^{2n-1} (\underline{w}) \times L_k^{2n-1} (\underline{w}) \times \mathbb{R} \,,\;
i_2 \big( t, (x, x, 0) \big) = \big( x, \phi_t (x), g_t(x) \big) \,,
$$
where $g_t$ is the conformal factor of $\phi_t$.
Let $H_t$ be the contact Hamiltonian function
of the contact isotopy $\{\phi_t\}$.
Then
$$
H_t \big(\phi_t(x)\big)
= (\pi_2^{\ast} \alpha_0 - e^{\theta} \pi_1^{\ast} \alpha_0)
\Big( \frac{d}{dt} \, i_2 \big( t, (x, x, 0) \big) \Big)
= (dz - \lambda_{\can})
\Big( \frac{d}{dt} \, i_1 \big( t, (x, 0, 0) \big) \Big)
= f(x) \,,
$$
and so $\lvert H_t \lvert < \epsilon$.
Moreover
$[ \{ \phi_t \}_{t \in [0, 1]} ] = \widetilde{\phi}$,
because $\widetilde{\phi} \cdot [ \{ \phi_t \}_{t \in [0, 1]} ]^{-1}$
can be represented by a loop in $\mathcal{U}$,
which is simply connected.
We thus conclude that
$\nu_{\alpha} (\widetilde{\phi}) < \epsilon$,
as we wanted.
 
Using \eqref{equation: continuity in section 3}
and the fact that $\nu_{\alpha}$ is $\mathcal{C}^1$-continuous
we now deduce that each $c_j$ is $\mathcal{C}^1$-continuous.
Let $\widetilde{\phi} \in \widetilde{\Cont_0} \big( L_k^{2n-1} (\underline{w}), \xi_0 \big)$.
By $\mathcal{C}^1$-continuity of $\nu_{\alpha}$,
for any $\epsilon > 0$ there is a $\mathcal{C}^1$-neighborhood
$\widetilde{\mathcal{U}}$ of $\widetilde{\id}$ in $\widetilde{\Cont_0} \big( L_k^{2n-1} (\underline{w}), \xi_0 \big)$
such that $\left. \nu_{\alpha} \right\lvert_{\widetilde{\mathcal{U}}} < \epsilon$.
Then $\widetilde{\mathcal{V}} := \widetilde{\mathcal{U}} \cdot \widetilde{\phi}$
is a $\mathcal{C}^1$-neighborhood of $\widetilde{\phi}$
such that for every $\widetilde{\psi} \in \widetilde{\mathcal{V}}$
we have
$$
\nu_{\alpha} ( \widetilde{\phi} \cdot \widetilde{\psi}^{-1} )
= \nu_{\alpha} ( \widetilde{\psi} \cdot \widetilde{\phi}^{-1} )
< \epsilon \,.
$$
This implies that
there is a contact Hamiltonian function $H_t$
whose flow represents $\widetilde{\phi} \cdot \widetilde{\psi}^{-1}$
and satisfies $\int_0^1 \max \, \lvert H_t\rvert \, dt < \epsilon$.
Using \eqref{equation: continuity in section 3}
we thus conclude that for every $\widetilde{\psi} \in \widetilde{\mathcal{V}}$
we have $\big\lvert\, c_j (\widetilde{\phi}) - c_j (\widetilde{\psi}) \,\big\rvert < \epsilon$,
and so that $c_j$ is $\mathcal{C}^1$-continuous.

\subsection*{Triangle inequality}

We have to prove that
if either $k$ is even or $j$ is even
then
\[
c_{j + l} (\widetilde{\phi} \cdot \widetilde{\psi})
\leq c_j (\widetilde{\phi}) + \left\lceil c_l (\widetilde{\psi}) \right\rceil_{T_{\underline{w}}} \,.
\]
Let $T = c_j (\widetilde{\phi})$
and $N = \lceil c_l (\widetilde{\psi}) \rceil_{T_{\underline{w}}}$.
Since $\Pi: \widetilde{\Cont_0} \big( L_k^{2n-1} (\underline{w}), \xi_0 \big)
\rightarrow \Cont_0 \big( L_k^{2n-1} (\underline{w}), \xi_0 \big)$
is a local homeomorphism,
by the continuity property of the spectral selectors
and \autoref{lemma: non-degenerate}
we can assume that $\Pi (\widetilde{\phi})$ is non-degenerate
with respect to $\alpha_0$.
By the relation with translated points property,
the spectral selectors of $\widetilde{\phi}$ are thus all distinct,
and so $\mu \big( \widetilde{r_{-T}} \cdot \widetilde{\phi}\big) = - j$.
Using the fact that $\widetilde{r_{-N}}$ commutes with $\widetilde{\phi}$
(since $\{r_t\}$ has period $T_{\underline{w}}$)
and the triangle inequality for the non-linear Maslov index
(\autoref{proposition: nonlinear Maslov index general k}
(\ref{triangle inequality general Maslov}))
we thus have
\begin{equation*}
\mu \big( \widetilde{r_{- (T + N)}} \cdot \widetilde{\phi} \cdot \widetilde{\psi} \big)
= \mu \big( (\widetilde{r_{-T}} \cdot \widetilde{\phi}) \cdot (\widetilde{r_{-N}} \cdot \widetilde{\psi}) \big)
\leq \mu \big( \widetilde{r_{-T}} \cdot \widetilde{\phi}\big) + \mu \big( \widetilde{r_{-N}} \cdot \widetilde{\psi} \big)
\leq - (j + l) \,.
\end{equation*}
By definition of $c_{j + l}$ we conclude that
$$
c_{j + l} \big( \widetilde{\phi} \cdot \widetilde{\psi} \big) \leq T + N
= c_j (\widetilde{\phi}) + \left\lceil c_l (\widetilde{\psi}) \right\rceil_{T_{\underline{w}}} \,.
$$

\subsection*{Conjugation invariance}

We have to prove that
\begin{equation}\label{equation: conjugation invariance section 3}
\left\lceil c_j (\tpsi \cdot \widetilde{\phi} \cdot \tpsi^{-1}) \right\rceil_{T_{\underline{w}}}
= \left\lceil c_j (\widetilde{\phi}) \right\rceil_{T_{\underline{w}}} \,.
\end{equation}
Assume first that $\Pi(\widetilde{\phi})$ does not have discriminant points.
Let $\{\psi_t\}_{t \in [0, 1]}$ be a contact isotopy representing $\widetilde{\psi}$,
and consider the homotopy $\widetilde{\psi}_s = [ \{ \psi_{st} \}_{t \in [0, 1]} ]$
from $\widetilde{\psi}_0 = \widetilde{\id}$ to $\widetilde{\psi}_1 = \widetilde{\psi}$.
By the continuity property,
the map
$$
s \mapsto c_j (\widetilde{\psi}_s \cdot \widetilde{\phi} \cdot \widetilde{\psi}_s^{-1})
\in \bar{\mathcal{A}} (\widetilde{\psi}_s \cdot \widetilde{\phi} \cdot \widetilde{\psi}_s^{-1})
$$
is continuous.
Moreover,
$c_j (\widetilde{\psi}_s \cdot \widetilde{\phi} \cdot \widetilde{\psi}_s^{-1})
\in \mathbb{R} \smallsetminus T_{\underline{w}} \cdot \mathbb{Z}$
for all $s \in [0,1]$.
Indeed,
if we had
$c_j (\widetilde{\psi}_{\underline{s}} \cdot \widetilde{\phi} \cdot \widetilde{\psi}_{\underline{s}}^{-1})
\in T_{\underline{w}} \cdot \mathbb{Z}$
for some $\underline{s}$ then, by the spectrality property,
$\Pi (\widetilde{\psi}_{\underline{s}} \cdot \widetilde{\phi} \cdot \widetilde{\psi}_{\underline{s}}^{-1})$
would have discriminant points.
But this is absurd,
because the discriminant points of
$\Pi (\widetilde{\psi}_{\underline{s}} \cdot \widetilde{\phi} \cdot \widetilde{\psi}_{\underline{s}}^{-1})$
are in bijection with the discriminant points of  $\Pi (\widetilde{\phi})$.
We thus obtain \eqref{equation: conjugation invariance section 3}
in this case.

The general case can be obtained as follows.
Given any $\widetilde{\phi} \in \widetilde{\Cont_0} \big( L_k^{2n-1} (\underline{w}), \xi_0 \big)$,
since (by \autoref{lemma: closed})
$\mathcal{A} (\widetilde{\phi})$ is nowhere dense,
there is a sequence $(\epsilon_l)$ of positive real numbers
with $\epsilon_l \to 0$ such that,
for every $l$,
$\Pi (\widetilde{r_{- \epsilon_l}} \cdot  \widetilde{\phi})$
does not have discriminant points.
Pose $\widetilde{\chi_l} = \widetilde{r_{- \epsilon_l}} \cdot  \widetilde{\phi}$.
By the first part of the proof we have
\begin{equation}\label{equation: in proof conjugation invariance}
\left\lceil c_j (\tpsi \cdot \widetilde{\chi_l} \cdot \tpsi^{-1}) \right\rceil_{T_{\underline{w}}}
= \Big\lceil c_j (\widetilde{\chi}_l) \Big\rceil_{T_{\underline{w}}}
\end{equation}
for all $l$.
Since $(\widetilde{\chi}_l)$ converges to $\widetilde{\phi}$
in the $\mathcal{C}^1$-topology
and $\widetilde{\chi}_l \leq \widetilde{\phi}$
for all $l$,
by the continuity and monotonicity properties of the spectral selectors
for $l$ big enough we have
$\left\lceil c_j (\widetilde{\phi}) \right\rceil_{T_{\underline{w}}}
= \left\lceil c_j (\widetilde{\chi}_l) \right\rceil_{T_{\underline{w}}}$
and $\left\lceil c_j (\tpsi \cdot \widetilde{\phi} \cdot \tpsi^{-1}) \right\rceil_{T_{\underline{w}}}
= \left\lceil c_j (\tpsi \cdot \widetilde{\chi}_l \cdot \tpsi^{-1}) \right\rceil_{T_{\underline{w}}}$.
Equation \eqref{equation: in proof conjugation invariance}
thus gives the desired result
\eqref{equation: conjugation invariance section 3}.

\subsection*{Poincar\'e duality}

We first notice that
if $\Pi \big( \mathcal{\widetilde{\phi}} \big)$
does not have discriminant points
then
\begin{equation}\label{equation: max}
\left\lfloor c_j \big(\widetilde{\phi}\big) \right\rfloor_{T_{\underline{w}}}
= \max \left\{\, N \in T_{\underline{w}} \cdot \mathbb{Z}  \;\ \Big\lvert\;
\mu (\widetilde{r_{-N}} \cdot \widetilde{\phi}) > - j  \,\right\} \,.
\end{equation}
Indeed,
by spectrality we have
$c_j (\widetilde{\phi}) \notin T_{\underline{w}} \cdot \mathbb{Z}$
and thus $\underline{N} :=
\big\lfloor c_j \big(\widetilde{\phi}\big) \big\rfloor_{T_{\underline{w}}}  < c_j (\widetilde{\phi})$.
This implies that
$\mu (\widetilde{r_{- \underline{N}}} \cdot \widetilde{\phi}) > - j$,
and so the inequality $\leq$ in \eqref{equation: max}.
On the other hand,
the opposite inequality follows
(without any assumption on $\widetilde{\phi}$)
from \eqref{equation: min}
and the fact that
$\big\lfloor c_j \big(\widetilde{\phi}\big) \big\rfloor_{T_{\underline{w}}} + T_{\underline{w}}
\geq \big\lceil c_j \big(\widetilde{\phi}\big) \big\rceil_{T_{\underline{w}}}$.

We now prove the Poincar\'e duality property,
i.e.\ that 
\begin{equation}\label{equation: PD in section 3}
\left\lceil c_j (\widetilde{\phi}) \right\rceil_{T_{\underline{w}}} =
- \left\lfloor c_{- j - (2n-1)} (\widetilde{\phi}^{-1}) \right\rfloor_{T_{\underline{w}}}
\end{equation}
for any $\widetilde{\phi} \in \widetilde{\Cont_0} \big( L_k^{2n-1} (\underline{w}), \xi_0 \big)$.
Assume first that $\Pi \big( \mathcal{\widetilde{\phi}} \big)$
does not have discriminant points.
Then \autoref{proposition: nonlinear Maslov index general k}
(\ref{PD general Maslov})
implies that
\begin{equation}\label{equation: in PD Simon}
\mu ( \widetilde{r_{N}} \cdot \widetilde{\phi} )
+ \mu ( \widetilde{\phi}^{-1} \cdot  \widetilde{r_{-N}} ) = 2n
\end{equation}
for every $N$ that is a multiple of $T_{\underline{w}}$.
The Poincar\'e duality \eqref{equation: PD in section 3}
then follows from
\eqref{equation: min}, \eqref{equation: max},
\eqref{equation: in PD Simon}
and the fact that
$\widetilde{\phi}^{-1} \cdot  \widetilde{r_{-N}} = \widetilde{r_{-N}} \cdot \widetilde{\phi}^{-1}$
for every $N$ that is a multiple of $T_{\underline{w}}$.

For a general $\widetilde{\phi}$,
as in the proof of conjugation invariance
we can find a sequence $(\widetilde{\chi}_l)$
that converges to $\widetilde{\phi}$ in the $\mathcal{C}^1$-topology
and such that, for all $l$,
$\widetilde{\chi}_l \leq \widetilde{\phi}$
and $\Pi (\widetilde{\chi}_l)$ does not have discriminant points.
By the first part of the proof we have
$\left\lceil c_j (\widetilde{\chi}_l) \right\rceil_{T_{\underline{w}}} =
- \left\lfloor c_{- j - (2n-1)} (\widetilde{\chi}_l^{-1}) \right\rfloor_{T_{\underline{w}}}$.
By monotonicity and continuity of the spectral selectors
we thus obtain \eqref{equation: PD in section 3}
also in this case.


\section{Non-shortening of the standard Reeb flow with respect to the discriminant and oscillation norms}\label{section: Reeb geodesic}

Recall from \cite{CS} that the discriminant norm $\nu_{\dis}$
on the universal cover $\widetilde{\Cont_0} (M, \xi)$
of the identity component of the contactomorphism group
of a closed contact manifold $(M, \xi)$
is the word norm 
associated to the generating set $\mathcal{D}$
formed by elements $\widetilde{\phi}$
that can be represented by an embedded contact isotopy,
i.e.\ a contact isotopy $\{\phi_t\}_{t \in [0, 1]}$
such that $\phi_t \circ \phi_s^{-1}$ has no discriminant points
for all $ s \ne t \in [0,1]$.
Recall also from \cite{PA} that
the discriminant length of a contact isotopy $\{\phi_t\}_{t \in [0, 1]}$
is the minimal $N$
such that there is a decomposition
$0 = t_0 < \cdots < t_N = 1$ of the time interval $[0, 1]$
with $\{ \phi_t \}_{t \in [t_{j}, t_{j + 1}]}$ embedded
for all $j = 0, \cdots, N - 1$.

Consider the discriminant norm
on $\widetilde{\Cont_0} \big( L_k^{2n-1} (\underline{w}), \xi_0 \big)$.
For every positive real number $T$
we have
\begin{equation}\label{equation: estimate with mu}
\nu_{\dis} (\widetilde{r_{T}}) \geq \frac{2n \lceil \frac{T}{2\pi} \rceil + 1}{2n+1} \,.
\end{equation}
Indeed,
let $N = \nu_{\dis} (\widetilde{r_{T}})$
and write $\widetilde{r_{T}} = \prod_{j = 1}^N \widetilde{\phi}_j$
with $\widetilde{\phi}_j \in \mathcal{D}$.
Then, by \autoref{proposition: nonlinear Maslov index general k}
(\ref{Reeb flow general Maslov}),
(\ref{triangle inequality general Maslov}),
(\ref{examples general Maslov})
and the first statement of
(\ref{relation with translated points general Maslov})
we have
$$
2n \, \left\lceil \frac{T}{2\pi} \right\rceil = \mu (\widetilde{r_{T}})
\leq \sum_{j = 1}^N \mu (\widetilde{\phi_j}) + N - 1
\leq 2n N + N - 1 \,.
$$
Similarly,
in the case of projective space
we have
\begin{equation}\label{equation: estimate with mu projective space}
\nu_{\dis} (\widetilde{r_{T}}) \geq \left\lceil \frac{T}{2\pi} \right\rceil \,.
\end{equation}
The estimates \eqref{equation: estimate with mu}
and \eqref{equation: estimate with mu projective space}
are better than those obtained in \cite{CS} and \cite{GKPS},
since in those references
just the quasimorphism property of the non-liner Maslov index
(\autoref{proposition: nonlinear Maslov index general k}
(\ref{quasimorphism general Maslov}))
is used
and not the triangle inequality
(\autoref{proposition: nonlinear Maslov index general k}
(\ref{triangle inequality general Maslov})).
However,
they are still not optimal.
Indeed, writing
$$
0 < T_0 := \frac{T}{ \left\lfloor \frac{k}{2\pi} \, T \right\rfloor + 1} < \frac{2\pi}{k}
$$
we have
\begin{equation}\label{equation: decomposition}
\{r_{Tt}\}_{t \in [0, 1]} = \{ \underbrace{r_{T_0t} \circ \cdots \circ r_{T_0t}}_{ \left\lfloor \frac{k}{2\pi} \, T \right\rfloor + 1} \}_{t \in [0, 1]} \,.
\end{equation}
Since the minimal period of a closed Reeb orbit
of $\alpha_0$ on $L_k^{2n-1} (\underline{w})$ is $\frac{2\pi}{k}$
we have that $\{r_{T_0t}\}_{t \in [0, 1]} \in \mathcal{D}$,
and so $\{r_{Tt}\}_{t \in [0, 1]}$ has discriminant length
smaller or equal than $\left\lfloor \frac{k}{2\pi} \, T \right\rfloor + 1$;
this length is actually equal to $\left\lfloor \frac{k}{2\pi} \, T \right\rfloor + 1$,
because for any interval $[t_0, t_1]$ of length $t_1 - t_0 \geq \frac{2 \pi}{k}$
the contact isotopy $\{r_{t}\}_{t \in [t_0, t_1]}$ is not embedded.
For instance,
in the case of projective space the discriminant length
of $\{ r_{2 \pi m t} \}_{t \in [0, 1]}$
is thus $2m + 1$,
while \eqref{equation: estimate with mu projective space}
only gives $\nu_{\dis} (\widetilde{r_{2 \pi m}}) \geq m$.
In this section we prove the optimal estimates
for the discriminant and oscillation lengths
of the standard Reeb flow of lens spaces with equal weights
using the spectral selectors
defined in \autoref{section: action selectors}.
The main advantage of using the spectral selectors
is that while
(by \autoref{proposition: nonlinear Maslov index general k}
(\ref{relation with translated points general Maslov}))
the non-linear Maslov index only jumps in the presence
of discriminant points of the lift of a contact isotopy
of a lens space to the sphere,
so that in particular for instance
$\mu (\widetilde{r_{T}}) = 2n \, \left\lceil \frac{T}{2\pi} \right\rceil$,
the spectral selectors
allow to distinguish $\widetilde{r_{T}}$
also for different values of $T$ in $[0, 2\pi]$,
indeed by \autoref{theorem: main} (\ref{composition with the Reeb flow main})
we have for instance $c_0 (\widetilde{r_{T}}) = T$.

We start with the following lemma.

\begin{lemma}\label{lemma: D}
For any element $\widetilde{\phi}$ of $\widetilde{\Cont_0} \big( L_k^{2n-1} (\underline{w}), \xi_0 \big)$,
if $\widetilde{\phi} \in \mathcal{D}$
then $c_0 (\widetilde{\phi}) < T_{\underline{w}}$.
\end{lemma}

\begin{proof}
If $\widetilde{\phi} \in \mathcal{D}$
then $\widetilde{\phi}$ can be represented
by a contact isotopy $\{\phi_t \}_{t \in [0, 1]}$
such that $\phi_t$ does not have discriminant points
for all $t \in (0, 1]$.
Suppose by contradiction
that $c_0 (\widetilde{\phi}) \geq T_{\underline{w}}$,
and for $s \in [0, 1]$ let $\widetilde{\phi}_s = [ \{\phi_{st}\}_{t \in [0, 1]} ]$.
Since $c_0 (\widetilde{\phi}_1) = c_0 (\widetilde{\phi}) \geq T_{\underline{w}}$
and, by \autoref{theorem: main} (\ref{normalization}),
$c_0 (\widetilde{\phi}_0) = c_0 (\widetilde{\id}) = 0$,
by continuity of $c_0$
(\autoref{theorem: main} (\ref{continuity}))
there is a value of $s$ in $(0, 1]$
such that $c_0 (\widetilde{\phi}_s) = T_{\underline{w}}$.
But then,
by spectrality (\autoref{theorem: main} (\ref{spectrality})),
$T_{\underline{w}}$ belongs to $\bar{\mathcal{A}} (\widetilde{\phi}_s)$.
This means that $\phi_s$ has discriminant points,
which is a contradiction.
\end{proof}

We can now prove that for every real number $T$
the Reeb flow $\{r_{Tt}\}_{t \in [0, 1]}$ of the standard contact form $\alpha_0$
on a lens space of the form $L_k^{2n-1} (w, \cdots, w)$
is a geodesic for the discriminant norm.
We have seen above that $\{r_{Tt}\}_{t \in [0, 1]}$
has discriminant length $\left\lfloor \frac{k}{2\pi} \, T \right\rfloor + 1$.
In order to prove that it is a geodesic
we thus have to show that
\begin{equation}\label{equation: geodesic}
\nu_{\dis} (\widetilde{r_T}) \geq \left\lfloor \frac{k}{2\pi} \, T \right\rfloor + 1 \,.
\end{equation}
Let $\nu_{\dis} (\widetilde{r_T}) = N$,
and write $\widetilde{r_T} = \prod_{j = 1}^N \widetilde{\phi_j}$
with $\widetilde{\phi_j} \in \mathcal{D}$ for all $j$.
By \autoref{theorem: main} (\ref{composition with the Reeb flow main}),
(\ref{triangle inequality})
and Lemma \ref{lemma: D},
and since $T_{\underline{w}} = \frac{2\pi}{k}$ for $\underline{w} = (w, \cdots, w)$,
we then have
$$
T = c_0 (\widetilde{r_T})
\leq c_0 (\widetilde{\phi_1})
+ \sum_{j = 2}^N \left\lceil c_0 (\widetilde{\phi_j}) \right\rceil_{\frac{2\pi}{k}}
< N \, \frac{2\pi}{k} \,.
$$
This implies that
$\nu_{\dis} (\widetilde{r_T}) \geq \left\lfloor \frac{k}{2\pi} \, T \right\rfloor + 1$,
as we wanted.

We now show that the standard Reeb flow
$\{r_{Tt}\}_{t \in [0, 1]}$ on a lens space of the form $L_k^{2n-1} (w, \cdots, w)$
is a geodesic with respect to the oscillation norm.
Recall from \cite{CS} that the oscillation pseudonorm $\nu_{\osc}$
on the universal cover $\widetilde{\Cont_0} (M, \xi)$
of the identity component of the contactomorphism group
of a closed contact manifold $(M, \xi)$
is defined as follows.
Let $\mathcal{D}_+$ and $\mathcal{D}_-$ be the sets
of elements of $\widetilde{\Cont_0} (M, \xi)$
that can be represented respectively by an embedded
non-negative or non-positive contact isotopy.
It is proved in \cite{CS}
that every element $\widetilde{\phi}$
of $\widetilde{\Cont_0} (M, \xi)$
can be written as $\widetilde{\phi} = \prod_{j = 1}^N \widetilde{\phi_j}$
with $\widetilde{\phi_j} \in \mathcal{D}_+$ or $\widetilde{\phi_j} \in \mathcal{D}_-$ for every $j$.
We denote by $\nu_+ (\widetilde{\phi})$ and $\nu_- (\widetilde{\phi})$
respectively the minimal number of elements of $\mathcal{D}_+$
and minus the minimal number of elements of $\mathcal{D}_-$
in such a decomposition.
The oscillation pseudonorm is then defined by
$$
\nu_{\osc} (\widetilde{\phi}) = \nu_+ (\widetilde{\phi}) - \nu_- (\widetilde{\phi})
$$
for $\widetilde{\phi} \neq \widetilde{\id}$,
and $\nu_{\osc} (\widetilde{\id}) = 0$.
By \cite[Proposition 3.2]{CS},
the oscillation pseudonorm on $\widetilde{\Cont_0} (M, \xi)$
is non-degenerate if and only if $(M, \xi)$ is orderable;
it is thus a norm for lens spaces.
Recall also from \cite{PA}
that the oscillation length of a contact isotopy $\{\phi_t\}_{t \in [0, 1]}$
is the sum of $\mathcal{L}_+ \big(\{\phi_t\}_{t \in [0, 1]}\big)$
and $\mathcal{L}_- \big(\{\phi_t\}_{t \in [0, 1]}\big)$,
where $\mathcal{L}_+ \big(\{\phi_t\}_{t \in [0, 1]}\big)$
is the minimal $N_+$ for which there is $N \geq N_+$
and a decomposition
$0 = t_0 < \cdots < t_N = 1$
with each $\{ \phi_t \}_{t \in [t_{j}, t_{j + 1}]}$ embedded
and non-negative or non-positive
and exactly $N_+$ of them non-negative,
and $\mathcal{L}_- \big(\{\phi_t\}_{t \in [0, 1]}\big)$
is the minimal $N_-$ for which there is $N \geq N_-$
and a decomposition
$0 = t_0 < \cdots < t_N = 1$
with each $\{ \phi_t \}_{t \in [t_{j}, t_{j + 1}]}$ embedded
and non-negative or non-positive
and exactly $N_-$ of them non-positive.

Consider now as above the class $\widetilde{r_T} = [ \{ r_{Tt} \}_{t \in [0, 1]} ]$
of the Reeb flow of $\alpha_0$
on a lens space $\big( L_k^{2n-1} (\underline{w}), \xi_0 \big)$.
Similarly as before,
the oscillation length of $\{ r_{Tt} \}_{t \in [0, 1]}$
is $\left\lfloor \frac{k}{2\pi} \, T \right\rfloor + 1$.
The decomposition \eqref{equation: decomposition}
shows moreover that $\nu_- (\widetilde{r_T}) = 0$,
and thus $\nu_{\osc} (\widetilde{r_T}) = \nu_+ (\widetilde{r_T})$.
In order to show that the Reeb flow $\{ r_{Tt} \}_{t \in [0, 1]}$
of $\alpha_0$ on a lens space of the form $L_k^{2n-1} (w, \cdots, w)$
is a geodesic with respect to the oscillation norm
we thus have to show that 
\[
\nu_+ (\widetilde{r_T}) \geq \left\lfloor \frac{k}{2\pi} \, T \right\rfloor + 1 \,.
\]
Let $\nu_+ (\widetilde{r_T}) = N_+$,
and write $\widetilde{r_T} = \prod\limits_{j = 1}^N \widetilde{\phi}_j$
with $\widetilde{\phi}_j \in \mathcal{D}_\pm$ for all $j$
and with exactly $N_+$ of the $\widetilde{\phi_j}$ in $\mathcal{D}_+$.
Denote such elements by
$\widetilde{\phi_{\sigma(1)}}$, $\cdots$, $\widetilde{\phi_{\sigma(N_+)}}$.
Then $\widetilde{r_T} \leq \prod\limits_{j = 1}^{N_+} \widetilde{\phi}_{\sigma(j)}$,
and so by \autoref{theorem: main} (\ref{composition with the Reeb flow main}),
(\ref{monotonicity}), (\ref{triangle inequality})
and Lemma \ref{lemma: D},
and since $T_{\underline{w}} = \frac{2\pi}{k}$ for $\underline{w} = (w, \cdots, w)$,
we have
\[
T = c_0 (\widetilde{r_T})
\leq c_0 \Big( \prod\limits_{j = 1}^{N_+} \widetilde{\phi}_{\sigma(j)} \Big)
\leq c_0 (\widetilde{\phi}_{\sigma(1)})
+ \sum _{j = 2}^{N_+} \left\lceil c_0 (\widetilde{\phi}_{\sigma(j)})\right\rceil_{\frac{2\pi}{k}}
< \frac{2\pi}{k} \cdot N_+  \,.
\]
This implies that
$\nu_+ (\widetilde{r_T}) \geq \left\lfloor \frac{k}{2\pi} \, T \right\rfloor + 1$,
as we wanted.

\begin{rmk}\label{remark: general lens spaces}
For a general lens space $\big( L_k^{2n-1} (\underline{w}), \xi_0 \big)$
the above discussion implies that
the discriminant and oscillation norms
of $\widetilde{r_T} = [ \{ r_{Tt} \}_{t \in [0, 1]} ]$
are greater or equal than $\left\lfloor \frac{T}{T_{\underline{w}}} \right\rfloor + 1$,
while the discriminant and oscillation lengths of $\{ r_{Tt} \}_{t \in [0, 1]}$
are equal to $\left\lfloor \frac{k}{2\pi} \, T \right\rfloor + 1$.
We do not know thus
if the standard Reeb flow on general lens spaces
is a geodesic for the discriminant and oscillation norms.
The gap for general weights between the minimal period $\frac{2 \pi}{k}$ of a closed Reeb orbit
and the period $T_{\underline{w}}$ of the Reeb flow
seems to suggest that it might be possible to shorten the Reeb flow.
It would be interesting to investigate if this is indeed the case,
or if other methods could be used to prove
that the Reeb flow is still a geodesic.
\end{rmk}


\section{A spectral pseudonorm}\label{section: conjugation invariant norms}

Let $c_- = c_{-2n+1}$ and $c_+ = c_0$,
and define
$\nu: \widetilde{\Cont_0} \big( L_k^{2n-1} (\underline{w}), \xi_0 \big) \rightarrow T_{\underline{w}} \cdot \mathbb{Z}$
by
$$
\nu (\widetilde{\phi})
= \max \left\{\, \left\lceil c_+ (\widetilde{\phi}) \right\rceil_{T_{\underline{w}}} \,,\,
-  \left\lfloor c_- (\widetilde{\phi}) \right\rfloor_{T_{\underline{w}}} \,\right\} \,.
$$
In this section we prove that $\nu$ is a pseudonorm
satisfying the properties
stated in \autoref{corollary: word norms}.

Recall that a pseudonorm $\nu$ on a group $G$
is said to be stably unbounded
if there is an element $\sigma$ of $G$
such that $\lim_{m \to \infty} \frac{\nu(\sigma^m)}{m}\neq 0$,
and is said to be compatible with a bi-invariant partial order $\leq$
if $\id \leq \sigma_1 \leq \sigma_2$
implies $\nu (\sigma_1) \leq \nu (\sigma_2)$.

\begin{prop}
The map $\nu: \widetilde{\Cont_0} \big( L_k^{2n-1} (\underline{w}), \xi_0 \big) \rightarrow T_{\underline{w}} \cdot \mathbb{Z}$
is a stably unbounded conjugation invariant pseudonorm
compatible with the partial order $\leq$.
\end{prop}

\begin{proof}
We first show that for every $\widetilde{\phi}$ we have $\nu (\widetilde{\phi}) \geq 0$.
Suppose by contradiction that $\nu (\widetilde{\phi}) < 0$.
Then $\left\lceil c_+ (\widetilde{\phi}) \right\rceil_{T_{\underline{w}}} < 0$,
thus $c_+ (\widetilde{\phi}) < 0$,
and $-  \left\lfloor c_- (\widetilde{\phi}) \right\rfloor_{T_{\underline{w}}} < 0$,
thus $c_- (\widetilde{\phi}) > 0$.
But this contradicts the fact that,
since the sequence $c_j$ is non-decreasing,
$c_- (\widetilde{\phi}) \leq c_+ (\widetilde{\phi})$.
The triangle inequality
$\nu (\widetilde{\phi} \cdot \widetilde{\psi})
\leq \nu (\widetilde{\phi}) + \nu (\widetilde{\psi})$
follows from 
\autoref{theorem: main} (\ref{triangle inequality}) and (\ref{PD}),
and symmetry $\nu (\widetilde{\phi}) = \nu (\widetilde{\phi}^{-1})$
from \autoref{theorem: main} (\ref{PD}).
This shows that $\nu$ is a pseudonorm.
Invariance by conjugation follows from
\autoref{theorem: main} (\ref{conjugaison invariance}) and (\ref{PD}).
The pseudonorm $\nu$ is stably unbounded,
indeed  \autoref{theorem: main}
(\ref{composition with the Reeb flow main})
implies that
$$
\nu \Big( {\widetilde{r_{T_{\underline{w}}}} }^m \Big)
= \nu \big(\widetilde{r_{m T_{\underline{w}}}} \big) = m \, T_{\underline{w}}
$$
for every positive integer $m$,
thus posing $\sigma = \widetilde{r_{T_{\underline{w}}}}$
we have $\lim_{m \to \infty} \frac{ \nu(\sigma^m) }{m} = T_{\underline{w}} \neq 0$.
Finally, the fact that $\nu$ is compatible with the partial order $\leq$
follows from \autoref{theorem: main} (\ref{monotonicity}) and (\ref{PD}).
\end{proof}

It would be interesting to know
if $\nu$ is equivalent to the oscillation norm $\nu_{\osc}$.
In this direction,
we prove the following inequality.

\begin{prop}\label{proposition: comparison}
For every element $\widetilde{\phi}$ of $\widetilde{\Cont_0} \big( L_k^{2n-1} (\underline{w}), \xi_0 \big)$
we have
$$
\nu (\widetilde{\phi}) \leq T_{\underline{w}} \cdot \nu_{\osc} (\widetilde{\phi}) \,.
$$
\end{prop}

\begin{proof}
Let $\nu_+ (\widetilde{\phi}) = N_+$,
and write $\widetilde{\phi} = \prod\limits_{j = 1}^N \widetilde{\phi}_j$
with all the $\widetilde{\phi}_j$ in $\mathcal{D}_+$ or $\mathcal{D}_-$
and exactly $N_+$ of them in $\mathcal{D}_+$.
Denote such elements by
$\widetilde{\phi_{\sigma(1)}}$, $\cdots$, $\widetilde{\phi_{\sigma(N_+)}}$.
Then $\widetilde{\phi} \leq \prod_{j = 1}^{N_+} \widetilde{\phi}_{\sigma(j)}$,
and thus by \autoref{theorem: main} (\ref{monotonicity}),
(\ref{triangle inequality})
and \autoref{lemma: D}
we have
$$
c_+ (\widetilde{\phi}) \leq c_+ \Big( \prod_{j = 1}^{N_+} \widetilde{\phi}_{\sigma(j)} \Big)
\leq \sum \limits_{j = 1}^{N_+} \left\lceil c_+ (\widetilde{\phi}_{\sigma (j)}) \right\rceil_{T_{\underline{w}}}
\leq T_{\underline{w}} \cdot N_+ \,.
$$
Similarly,
setting $\nu_- (\widetilde{\phi}) = - N_-$
we have $c_+ (\widetilde{\phi}^{-1}) \leq T_{\underline{w}} \cdot N_-$
and so, by \autoref{theorem: main} (\ref{PD}),
$$
- \left\lfloor c_- (\widetilde{\phi}) \right\rfloor_{T_{\underline{w}}}
= \left\lceil c_+ \left(\widetilde{\phi}^{-1}\right)\right\rceil_{T_{\underline{w}}}
\leq T_{\underline{w}} \cdot N_- \,.
$$
We deduce that 
\[
\nu (\widetilde{\phi})
\leq T_{\underline{w}} \, \max \Big\{ \nu_+ (\widetilde{\phi}) \,,\, - \nu_- (\widetilde{\phi}) \Big\}
\leq T_{\underline{w}} \big( \nu_+ (\widetilde{\phi}) - \nu_- (\widetilde{\phi}) \big)
= T_{\underline{w}} \cdot \nu_{\osc} (\widetilde{\phi}) \,.
\]
\end{proof}

We do not know whether the pseudonorm $\nu$ is non-degenerate,
i.e.\ whether $\nu (\widetilde{\phi}) = 0$
if and only if $\widetilde{\phi} = \widetilde{\id}$.
Indeed,
by the definition of $\nu$ we have that
$\nu (\widetilde{\phi}) = 0$
if and only if $c_+ (\widetilde{\phi}) = c_- (\widetilde{\phi}) = 0$,
which by \autoref{theorem: main} (\ref{non-degeneracy})
only implies that $\Pi (\widetilde{\phi})$ is the identity.
On the other hand,
the induced conjugation invariant pseudonorm on $\Cont_0 \big( L_k^{2n-1} (\underline{w}), \xi_0 \big)$,
i.e.\ the pseudonorm $\nu_{\ast}$ defined by
$$
\nu_{\ast} (\phi) = \inf \{\, \nu (\widetilde{\phi}) \;\lvert\; \Pi (\widetilde{\phi}) = \phi  \,\} \,,
$$
is non-degenerate, hence a norm.
However, this norm is bounded
(hence equivalent to the trivial norm,
since it is discrete),
as shown in the following proposition.

\begin{prop}\label{lemma: nu bounded}
For every $\phi \in \Cont_0 \big( L_k^{2n-1} (\underline{w}), \xi_0 \big)$
we have $\nu_{\ast} (\phi) \leq 2 \pi + T_{\underline{w}}$.
\end{prop}

\begin{proof}
We show that
\begin{equation}\label{equation: in proof nu bounded}
\nu_{\ast} \big( \Pi (\widetilde{\phi}) \big)
\leq  \left\lceil c_+ (\widetilde{\phi}) \right\rceil_{T_{\underline{w}}}
-  \left\lfloor c_- (\widetilde{\phi}) \right\rfloor_{T_{\underline{w}}}
\leq 2 \pi + T_{\underline{w}}
\end{equation}
for every $\widetilde{\phi} \in \widetilde{\Cont_0} \big( L_k^{2n-1} (\underline{w}), \xi_0 \big)$.
Using periodicity of the spectral selectors
(\autoref{theorem: main} (\ref{periodicity}))
and the fact that the sequence of spectral selectors $c_j$ is non-decreasing
we have $c_+ (\widetilde{\phi}) \leq c_- (\widetilde{\phi}) + 2\pi$,
which implies the second inequality
in \eqref{equation: in proof nu bounded}.
For the first inequality,
it is enough to find $N \in T_{\underline{w}} \cdot \mathbb{Z}$
such that $\nu (\widetilde{r_{- N}} \cdot \widetilde{\phi})
= \left\lceil c_+ (\widetilde{\phi}) \right\rceil_{T_{\underline{w}}}
-  \left\lfloor c_- (\widetilde{\phi}) \right\rfloor_{T_{\underline{w}}}$.
Suppose first that
 $\nu (\widetilde{\phi}) =  \left\lceil c_+ (\widetilde{\phi}) \right\rceil_{T_{\underline{w}}}$,
 and pose $N =  \left\lceil c_+ (\widetilde{\phi}) \right\rceil_{T_{\underline{w}}}$.
By \autoref{theorem: main} (\ref{composition with the Reeb flow main})
we then have
$$
\nu (\widetilde{r_{- N}} \cdot \widetilde{\phi})
= \max \left\{\, \left\lceil c_+ (\widetilde{r_{- N}} \cdot\widetilde{\phi}) \right\rceil_{T_{\underline{w}}} \,,\,
-  \left\lfloor c_- (\widetilde{r_{- N}} \cdot\widetilde{\phi}) \right\rfloor_{T_{\underline{w}}} \,\right\}
$$
$$
= \max \left\{\, 0 \,,\, \left\lceil c_+ (\widetilde{\phi}) \right\rceil_{T_{\underline{w}}}
- \left\lfloor c_- (\widetilde{\phi}) \right\rfloor_{T_{\underline{w}}} \,\right\}
= \left\lceil c_+ (\widetilde{\phi}) \right\rceil_{T_{\underline{w}}}
-  \left\lfloor c_- (\widetilde{\phi}) \right\rfloor_{T_{\underline{w}}} \,.
$$
Similarly,
if $\nu (\widetilde{\phi}) =  - \left\lfloor c_- (\widetilde{\phi}) \right\rfloor_{T_{\underline{w}}}$
then, posing $N = \left\lfloor c_- (\widetilde{\phi}) \right\rfloor_{T_{\underline{w}}}$
we have
$$
\nu (\widetilde{r_{- N}} \cdot \widetilde{\phi})
= \left\lceil c_+ (\widetilde{\phi}) \right\rceil_{T_{\underline{w}}}
-  \left\lfloor c_- (\widetilde{\phi}) \right\rfloor_{T_{\underline{w}}} \,.
$$
\end{proof}

\begin{rmk}
It follows from \cite[Corollary 4.12]{AA} that
on the universal cover of the identity component
of the contactomorphism group
of the unit cotangent bundle of the torus $\mathbb{T}^n$
for $n \geq 2$
the difference of the invariants $c_+$ and $c_-$
defined in \cite{AA}
is unbounded.
This difference with respect to \eqref{equation: in proof nu bounded}
might be related to the fact that
the identity component
of the contactomorphism group
of the unit cotangent bundle of the torus
does not contain positive loops.
It would be interesting to investigate
if on the other hand the difference of the invariants $c_+$ and $c_-$ of \cite{AA}
is bounded on $\widetilde{\Cont_0} \big( L_k^{2n-1} (\underline{w}), \xi_0 \big)$.
This would then imply as in \autoref{lemma: nu bounded}
that the induced norm on $\Cont_0 \big( L_k^{2n-1} (\underline{w}), \xi_0 \big)$
is bounded,
and therefore answer partially a question in \cite[Example 2.21]{FPR}.
\end{rmk}

\begin{rmk}
If $\nu : G \to \mathbb{R}_{\geq 0}$ is a pseudonorm on a group $G$
then, for any $c > 0$,
the map $\nu' : G \to \mathbb{R}_{\geq 0}$ defined by
\[
\nu' (g) := \left\{
\begin{array}{ll}
\max\{\nu(g), c \} & \mbox{ if } g \ne \id \\
0 & \mbox{ if } g = \id
\end{array}
\right.
\]
is a norm.
Moreover,
$\nu'$ is invariant by conjugation
if and only if so is $\nu$.
This trick
(which is similar to one used in \cite{BIP})
can be applied to our pseudonorm $\nu$,
with $c = T_{\underline{w}}$,
to obtain a stably unbounded conjugation invariant norm $\nu'$
on $\widetilde{\Cont_0} \big( L_k^{2n-1} (\underline{w}), \xi_0 \big)$.
Since $\nu$ takes values in $T_{\underline{w}} \cdot \mathbb{Z}$,
if $\nu$ is already a norm then $\nu' \equiv \nu$.
\autoref{proposition: comparison}
holds also for $\nu'$,
indeed for any element $\widetilde{\phi} \ne \widetilde{\id}$
we have 
\[
\nu' (\widetilde{\phi})
= \max \left\{ \left\lceil c_+ (\widetilde{\phi}) \right\rceil_{T_{\underline{w}}} ,
- \left\lfloor c_- (\widetilde{\phi}) \right\rfloor_{T_{\underline{w}}}, T_{\underline{w}} \right\}
\leq T_{\underline{w}} \cdot \max \big\{ \nu_{\osc} (\widetilde{\phi}), 1 \big\}
= T_{\underline{w}} \cdot \nu_{\osc} (\widetilde{\phi}) \,.
\]
\end{rmk}

\bibliographystyle{amsplain}
\bibliography{biblio}

\end{document}